\documentclass{amsart}
\usepackage{amsmath}
\usepackage{amsthm}
\usepackage{amssymb}
\usepackage{cite}
\usepackage[all]{xypic}

\newcommand{\IR}{\mathbb{R}}
\newcommand{\IC}{\mathbb{C}}
\newcommand{\IZ}{\mathbb{Z}}
\newcommand{\IN}{\mathbb{N}}
\newcommand{\ID}{\mathbb{D}}
\newcommand{\IT}{\mathbb{T}}
\newcommand{\IP}{\mathbb{P}}
\newcommand{\RE}{{\rm Re} \,}

\newcommand{\id}{{id}}

\newtheorem{theorem}{Theorem}[section]
\newtheorem{lemma}[theorem]{Lemma}
\newtheorem{corollary}[theorem]{Corollary}

\theoremstyle{definition}

\newtheorem{example}[theorem]{Example}

\theoremstyle{remark}

\numberwithin{equation}{section}

\begin{document}
\title[Fixed Point Composition and Toeplitz-Composition C*-algebras]{Fixed Point Composition and Toeplitz-Composition C*-algebras}
\author{Katie S. Quertermous}
\address{Katie S. Quertermous, Department of Mathematics \& Statistics, James Madison University, Harrisonburg, VA 22807}

\date{May 25, 2012}

\begin{abstract}  Let $\varphi$ be a linear-fractional, non-automorphism self-map of $\mathbb{D}$ that fixes $\zeta \in \mathbb{T}$ and satisfies $\varphi^{\prime}(\zeta) \neq 1$ and consider the composition operator $C_{\varphi}$ acting on the Hardy space $H^2(\mathbb{D}).$  We determine which linear-fractionally-induced composition operators are contained in the unital C$^*$-algebra generated by $C_{\varphi}$ and the ideal $\mathcal{K}$ of compact operators.  We apply these results to show that $C^*(C_{\varphi}, \mathcal{K})$ and $C^*(\mathcal{F}_{\zeta})$, the unital C$^*$-algebra generated by all composition operators induced by linear-fractional, non-automorphism self-maps  of $\mathbb{D}$ that fix $\zeta$, are each isomorphic, modulo the ideal of compact operators, to a unitization of a crossed product of $C_0([0,1])$.  We compute the K-theory of $C^*(C_{\varphi}, \mathcal{K})$ and calculate the essential spectra of a class of operators in this C$^*$-algebra.  We also obtain a full description of the structures, modulo the ideal of compact operators,  of the C$^*$-algebras generated by the unilateral shift $T_z$ and a single linear-fractionally-induced composition operator.\end{abstract}

\email{querteks@jmu.edu}
\maketitle

\section{Introduction}\label{intro}
For any analytic self-map $\varphi$ of the unit disk $\ID$, one can define the composition operator $C_{\varphi}: f \mapsto f \circ \varphi$, which is a bounded operator on the Hardy space $H^2(\ID)$.
In recent work, several authors have investigated unital C$^*$-algebras generated by linear-fractionally-induced composition operators and either the unilateral shift $T_z$ on $H^2(\ID)$ or the ideal $\mathcal{K}$ of compact operators on $H^2(\ID)$  \cite{Jury:2007G, Jury:2007F, KrieteMacCluerMoorhouse:2007, KrieteMacCluerMoorhouse:2009, KrieteMacCluerMoorhousePP}.  The key goals of this line of research are to describe the structures of these C$^*$-algebras, modulo the ideal of compact operators, and to use the structure results to determine spectral information for algebraic combinations of composition operators.  This approach to studying composition operators is motivated, in part, by  the well-known results for the shift algebra $C^*(T_z)$ that use a similar approach to determine spectral properties of Toeplitz operators with continuous symbols \cite{Coburn:1967, Coburn:1969}.

If $\varphi$ is a linear-fractional map, then $C_{\varphi}$ is a compact operator on $H^2(\ID)$ if and only if $||\varphi||_{\infty}:=\sup\{|\varphi(z)| : z \in \ID\} <1.$  If $\varphi$ is an automorphism of $\ID$, then, by the work of Jury, $C^*(T_z, C_{\varphi})/\mathcal{K}$ is isomorphic to a crossed product C$^*$-algebra of the form $C(\IT) \rtimes_{\alpha} \IZ$, if $\varphi$ is of infinite order, or the form $C(\IT) \rtimes_{\alpha} \IZ/q\IZ$, if $\varphi$ is of finite order $q$  \cite{Jury:2007F}.  Here $\IT$ denotes the unit circle, and the action $\alpha$ comes from the map $\varphi$.  Moreover, if $G$ is a discrete, abelian group of automorphisms of $\ID$, then $C^*(T_z, \{C_{\varphi}: \varphi \in G\})/\mathcal{K}$ is  isomorphic to a crossed product of $C(\IT)$ by $G$ \cite{Jury:2007G}. 

If $||\varphi||_{\infty} =1$ but $\varphi$ is not an automorphism of $\ID$, then the structure of the unital C$^*$-algebra generated by $C_{\varphi}$ and either $T_z$ or $\mathcal{K}$ depends on the configuration of the fixed points of $\varphi$.  There are three possibilities \cite{BLNS:2003, Shapiro:1993}:  \begin{enumerate} \item $\varphi$ fixes a point $\zeta \in \IT$ and has no other fixed points, i.e.\ $\varphi$ is a parabolic map.
\item $\varphi$ fixes a point $\zeta \in \IT$ and has an additional fixed point in $\IC \cup \{\infty\}$.
\item $\varphi$ has no fixed point in $\IT$.  In this case, there exist distinct points $\zeta, \eta \in \IT$ with $\varphi(\zeta)=\eta.$\end{enumerate}
One can distinguish between the first two cases by the value of $\varphi^{\prime}(\zeta).$  In both cases, $\varphi^{\prime}(\zeta)>0$, but $\varphi$ is parabolic if and only if $\varphi^{\prime}(\zeta)=1.$
Kriete, MacCluer, and Moorhouse have studied C$^*$-algebras related to non-automorphisms of the first and third types.  If  $\varphi$ is a parabolic, non-automorphism self-map of $\ID$, then $\mathcal{K} \subset C^*(C_{\varphi})$, and $C^*(C_{\varphi})/\mathcal{K}$ is isomorphic to $C([0,1])$ \cite{KrieteMacCluerMoorhousePP}.  For  a non-automorphism $\varphi$  of the third type, $C^*(T_z, C_{\varphi})/\mathcal{K}$ is isomorphic to a C$^*$-subalgebra of $C(\IT) \oplus M_2(C([0,1]))$ \cite{KrieteMacCluerMoorhouse:2007}.  A related description  for $C^*(T_z, C_{\varphi_1}, \ldots, C_{\varphi_n})/\mathcal{K}$ under certain conditions on the boundary points $\zeta_1, \ldots, \zeta_n$ and $\eta_1, \ldots, \eta_n$ was obtained in \cite{KrieteMacCluerMoorhouse:2009}.  We note that, in general, the work on the non-automorphism cases has used different techniques than those employed in  the automorphism setting.

In this paper, we consider the second non-automorphism case.  More specifically, we investigate the  unital  C$^*$-algebras $C^*(C_{\varphi}, \mathcal{K})$ and $C^*(C_{\varphi_1}, \ldots, C_{\varphi_n}, \mathcal{K})$, where $\varphi, \varphi_1, \ldots, \varphi_n$ are linear-fractional, non-automorphism self-maps of 
$\ID$ that fix $\zeta \in \IT$ and have a first derivative at $\zeta$ that is not equal to $1$.  We also study the C$^*$-algebra $C^*(\mathcal{F}_{\zeta})$, where $\mathcal{F}_{\zeta}$ is the collection of all composition operators induced by linear-fractional, non-automorphism self-maps of $\ID$ that fix the point $\zeta$.  Our arguments make extensive use of  known results about various types of linear-fractionally-induced composition operators, including automorphism-induced composition operators.  We include
 a collection of these results, along with basic facts about crossed product C$^*$-algebras, in Section \ref{prelim}.

In Section \ref{whatsinsection}, we determine which  linear-fractionally-induced composition operators are contained in $C^*(C_{\varphi}, \mathcal{K})$. As a part of this work, we show that, for all $\zeta \in \IT$,   $C^*(\mathcal{F}_{\zeta})$ is equal to the unital C$^*$-algebra generated by $\mathcal{K}$ and the collection of all composition operators induced by linear-fractional, non-automorphism self maps of $\ID$ that fix $\zeta$ and are not parabolic. 

We then apply these results in Section \ref{structuresection} to identify $C^*(\mathcal{F}_{\zeta})$, $C^*(C_{\varphi}, \mathcal{K})$,  and $C^*(C_{\varphi_1}, \ldots, C_{\varphi_n}, \mathcal{K})$ as subalgebras of C$^*$-algebras generated by  composition operators induced by parabolic, non-automorphism self-maps of $\ID$ and  unitary operators induced by automorphisms of $\ID$.  We prove that these larger C$^*$-algebras are isomorphic, modulo the ideal of compact operators, to crossed products of $C([0,1])$ by discrete groups.  By studying the images of  $C^*(\mathcal{F}_{\zeta})/\mathcal{K}$ and $C^*(C_{\varphi}, \mathcal{K})/\mathcal{K}$ under these isomorphisms, we obtain our main structure results, which demonstrate that these C$^*$-algebras are each isomorphic to a unitization of a crossed product of $C_0([0,1])$ by an appropriate discrete group.  Here, $C_0([0,1])$ denotes the collection of all continuous functions on $[0,1]$ that vanish at $0.$  We also obtain a similar theorem for $C^*(C_{\varphi_1}, \ldots, C_{\varphi_n}, \mathcal{K})/\mathcal{K}$ under a reasonable assumption  on the values of $\varphi_1^{\prime}(\zeta), \ldots, \varphi_n^{\prime}(\zeta).$  

In Section \ref{applicationsection}, we use the crossed product structure of $C^*(C_{\varphi}, \mathcal{K})/\mathcal{K}$ to determine the K-theory of $C^*(C_{\varphi}, \mathcal{K})$ and show that the Fredholm index of every Fredholm operator in $C^*(C_{\varphi}, \mathcal{K})$ is zero.   We also calculate the essential spectra of  operators in this C$^*$-algebra having a specific form.

Finally, in Section \ref{includeshift}, we determine the structure of $C^*(T_z, C_{\varphi})/\mathcal{K}$, where $\varphi$ is a linear-fractional self-map of $\ID$ that fixes a point $\zeta \in \IT$.  We combine these results with the work of Jury \cite{Jury:2007F} and Kriete, MacCluer, and Moorhouse \cite{KrieteMacCluerMoorhousePP, KrieteMacCluerMoorhouse:2007} to obtain a full description of the structure of $C^*(T_z, C_{\varphi})/\mathcal{K}$ for any linear-fractional self-map $\varphi$ of $\ID.$

\section{Preliminaries}\label{prelim}

\subsection{Operators on $H^2(\ID)$ and Relations Between Them}

The Hardy space of the disk, $H^2(\ID),$ is the space of all analytic functions $f$ on $\ID$ whose power series $f(z)= \sum_{n=0}^{\infty} a_n z^n$ satisfy $$||f||_{H^2(\ID)}^2 := \sum_{n=0}^{\infty} |a_n|^2 < \infty.$$  All operators in this paper will act on $H^2(\ID)$.  If $A$ is a bounded linear operator on $H^2(\ID)$, we will denote the coset of $A$ in the Calkin algebra $\mathcal{B}(H^2(\ID))/\mathcal{K}$ by $[A]$.  For more information on the Hardy space, see \cite{CowenMacCluer:1995} and \cite{Duren:1970}.

  If the linear-fractional map $\varphi(z)=(az+b)/(cz+d)$ takes the unit disk into itself, then the Krein adjoint of $\varphi$, which is defined by $$\sigma(z)=\sigma_{\varphi}(z)=\frac{\overline{a}z-\overline{c}}{-\overline{b}z+\overline{d}},$$ is also a self-map of $\ID$.  If $\varphi$ is  not an automorphism of $\ID$ and $\varphi(\zeta)=\eta$ for some $\zeta, \eta \in \IT$, then $\sigma(\eta)=\zeta$, $|\sigma^{\prime}(\eta)|=|\varphi^{\prime}(\zeta)|^{-1}$, and $\sigma$ is not an automorphism of $\ID$.  Moreover, if $\varphi$ is a linear-fractional, non-automorphism self-map of $\ID$ that maps $\zeta \in \IT$ to a point in $\IT$, then, by \cite[Theorem 3.1]{KrieteMacCluerMoorhouse:2007}, there exists a compact  operator $K$ such that  \begin{align}\label{adjform} C_{\varphi}^* = |\varphi^{\prime}(\zeta)|^{-1} C_{\sigma} + K. \end{align}

  Although our work in motivated by an interest in maps that are not parabolic, we will make extensive use of parabolic-induced composition operators.  Every parabolic, linear-fractional self-map of $\ID$ that fixes the point $\zeta \in \IT$ has the form
$$\rho_{\zeta, a}(z)=\frac{(2-a)z + a \zeta}{-a\overline{\zeta}z + (2+a)}$$ for some $a \in \IC$ with $\RE a \geq 0$.  The map $\rho_{\zeta, a}$ is an 
automorphism of $\ID$ if and only if  $\RE a =0$.  The number $a$ is called the translation number of $\rho_{\zeta, a}$ since $\rho_{\zeta, a}$ 
is conjugate to translation  by $a$ on $\Upsilon:=\{ z \in \IC : \RE z >0\}$ via the map $z \mapsto (\zeta+z)/(\zeta-z).$  
For $\zeta \in \IT$, 
we let $\IP_{\zeta} := \{C_{\rho_{\zeta, a}} : a \in \Upsilon\}$ denote the set of all composition operators induced by parabolic, non-automorphism self-maps of $\ID$ that fix $\zeta$.  The structure of the unital C$^*$-algebra generated by  $\IP_{\zeta}$, modulo the ideal of compact operators, is described by the following theorem:

\begin{theorem}{\rm \cite[Theorem 3]{KrieteMacCluerMoorhousePP}}\label{paraboliciso} Let $\zeta \in \IT$.  Then $\mathcal{K} \subseteq C^*(\IP_{\zeta})$, and there is a unique $*$-isomorphism 
$\Gamma_{\zeta}: C([0,1]) \rightarrow C^*(\IP_{\zeta}) / \mathcal{K}$ such that $\Gamma_{\zeta}(x^a)= [C_{\rho_{\zeta, a}}]$ for all $a \in \Upsilon.$  Moreover, if $\rho$ is any parabolic, non-automorphism self-map of $\ID$ that fixes $\zeta$, then $C^*(C_{\rho})=C^*(\IP_{\zeta}).$ \end{theorem}

The next theorem explains two important relationships between the parabolic, non-automorphism self-maps of $\ID$ and the linear-fractional self-maps $\varphi$ of $\ID$ that satisfy $||\varphi||_{\infty}=1$.

      \begin{theorem} {\rm \cite[Proposition 1]{KrieteMacCluerMoorhouse:2009} }\label{parabolicandourmaps} Suppose $\varphi$ is a linear-fractional self-map of $\ID$ that satisfies $\varphi(\zeta)= \eta$ for some $\zeta, \eta \in \IT$, and let $\sigma$ be its Krein adjoint.
       \begin{enumerate} 
       \item If $\varphi$ is not an automorphism of $\ID$, then there exist unique, positive numbers $b$ and $c$ such that $\rho_{\eta, b}(\ID)=\varphi(\ID)$ and $\rho_{\zeta, c}(\ID)=\sigma(\ID).$  Moreover, $\varphi \circ \sigma=\rho_{\eta, 2b}$ and $\sigma \circ \varphi=\rho_{\zeta, 2c}$.
       \item If $a \in \Upsilon$ or $a=0$, then \begin{align*}\label{commute}C_{\varphi}C_{\rho_{\eta, a}} = C_{\rho_{\zeta, |\varphi^{\prime}(\zeta)|a}} C_{\varphi}.\end{align*}
 \end{enumerate} \end{theorem}

      If $\varphi$ is an analytic self-map of $\ID$, we denote by $U_{\varphi}$ the partial isometry that appears in the polar decomposition of $C_{\varphi}$.  If $\varphi$ is a linear-fractional map, then $U_{\varphi}$ is a unitary operator.

  For every function $f \in H^2(\ID)$, the radial limit $f(e^{i\theta}):=\lim_{r \rightarrow 0} f(re^{i\theta})$ exists almost everywhere on $\IT$, and one can view $H^2(\ID)$ as a subspace of $L^2(\IT).$  Let $P$ denote the orthogonal projection of $L^2(\IT)$ onto $H^2(\ID)$.  If $g \in L^{\infty}(\IT)$, the Toeplitz operator $T_g$ is defined on $H^2(\ID)$ by $T_gf=Pgf$ for all $f \in H^2(\ID).$  It is a well-known result  that $C^*(T_z) = C^*(\{T_f : f \in C(\IT)\})$ and $C^*(T_z)/\mathcal{K}$ is isomorphic to $C(\IT)$ \cite{Coburn:1967, Coburn:1969}.  There exists a vast literature on Toeplitz operators (see, for example, \cite{BottcherSilbermann:2006, Douglas:1998, Douglas:1973}), and we refer the reader to those sources for more information.

We will use Toeplitz operators and unitary operators $U_{\varphi}$ induced by automorphisms to   write the cosets of composition operators induced by non-automorphisms in forms that reveal the underlying structures of the C$^*$-algebras under consideration.  The following two results will be especially important in our investigations.

\begin{theorem}\label{comprel} {\rm \cite{KrieteMacCluerMoorhouse:2007}}
 Suppose $\varphi$ is a linear-fractional, non-automorphism self-map of $\ID$ that satisfies $\varphi(\zeta)=\zeta$ for some $\zeta \in \IT$.  If $w \in C(\IT)$, then there exists compact operators $K_1$ and $K_2$ such that $$C_{\varphi}T_w=T_wC_{\varphi}+ K_1 = w(\zeta)C_{\varphi}+K_2.$$ \end{theorem}

   \begin{theorem}{\rm \cite{BourdonMacCluer:2007, Jury:2007G}}\label{unitaryrel}    Suppose $\gamma$ is an automorphism of $\ID$.    \begin{enumerate}
   \item There exists a compact operator $K$ such that $$U_{\gamma} = T_{w_{\gamma}}C_{\gamma} + K,$$ where $w_{\gamma}(z)=\sqrt{1-|\gamma^{-1}(0)|^2}/|1-\overline{\gamma^{-1}(0)}z|$ for all $z \in \IT.$
   \item If $\gamma_1$ is also an automorphism of $\ID$, then there exists a compact operator $K^{\prime}$ such that $$U_{\gamma}U_{\gamma_1}=U_{\gamma_1 \circ \gamma} + K^{\prime}.$$
   \end{enumerate}\end{theorem}

   \subsection{Crossed Product C*-algebras}

We recall here the definition of a crossed product C$^*$-algebra.  We  restrict our attention to the case of a discrete group $G$.  For the general case, we refer the reader to \cite{Williams:2007}.

If $\mathcal{A}$ is a C$^*$-algebra, let ${\rm Aut}(\mathcal{A})$ be the group of all $*$-automorphisms of $\mathcal{A}$.    If  a map $\alpha: G \rightarrow {\rm Aut}(\mathcal{A})$, acting by $s \mapsto \alpha_s$, is a group homomorphism, then we say that $\alpha$ is an action of $G$ on $\mathcal{A}$, and the triple $(\mathcal{A}, G, \alpha)$ is called a C$^*$-dynamical system.  

To build a crossed product from the C$^*$-dynamical system, we consider $C_c(G, \mathcal{A})$, the collection of all finitely-supported functions from $G$ to $\mathcal{A}$.  Each function $f$ in $C_c(G, \mathcal{A})$ can be written as $f=\sum_{s \in G} A_s \chi_s$, where $A_s \in \mathcal{A}$ for all $s \in G$, $A_s=0$ for all but a finite number of values of $s$, and $\chi_s$ is the characteristic function of $\{s\}$.  Addition on $C_c(G, \mathcal{A})$ is defined pointwise, and the product and involution are given by the formulas \begin{equation*} \left(\sum_{t \in G} A_t \chi_t \right) \left(\sum_{s \in G} B_s \chi_s \right) = \sum_{s \in G} \left(\sum_{r \in G} A_r \alpha_r(B_{r^{-1}s})\right)\chi_s
 \end{equation*}
 and \begin{equation*}\left(\sum_{t \in G}A_t \chi_t\right)^*= \sum_{t \in G} \alpha_t\left(A^*_{t^{-1}}\right)\chi_t.\end{equation*}   Under these operations, $C_c(G, \mathcal{A})$ is a $*$-algebra. 
  
 A covariant representation of $(\mathcal{A}, G, \alpha)$ is a pair $(\pi, W)$ consisting of a representation $\pi: \mathcal{A} \rightarrow \mathcal{B(H)}$ and a unitary representation $W: G \rightarrow \mathcal{B(H)}$, $s \mapsto W_s$,  that satisfy \begin{equation} \pi(\alpha_s(A))=W_s\pi(A)W_s^* \end{equation} for all $s \in G$ and $A \in \mathcal{A}.$  Given a covariant representation $(\pi, W)$, the integrated form $\pi \rtimes W$ of $(\pi, W)$, which is defined on $C_c(G, \mathcal{A})$ by \begin{align*} \left(\pi \rtimes W\right) \left(\sum_{s \in G} A_s \chi_s\right)=\sum_{s \in G} \pi(A_s) W_s,\end{align*} is a representation of $C_c(G, \mathcal{A})$ on $\mathcal{H}$.  The universal norm on $C_{c}(G,\mathcal{A})$ is then defined by \begin{equation*}\label{uninorm} ||f|| := \sup \{||(\pi \rtimes W)(f)|| : (\pi, W) \, \, \, \text{is a covariant representation of} \, \, \, (\mathcal{A}, G, \alpha)\}\end{equation*} for all $f \in C_c(G, \mathcal{A}).$  The completion of $C_c(G, \mathcal{A})$ in the universal norm is the crossed product C$^*$-algebra $\mathcal{A} \rtimes_{\alpha} G$.

In this paper, we will consider C$^*$-algebras of the form $C^*(\mathcal{A}, \{W_s : s \in G\})$, where $\mathcal{A}$ is a unital C$^*$-subalgebra of $\mathcal{B(H)}$ for some Hilbert space $\mathcal{H}$, $G$ is a discrete, amenable group, and  $s \mapsto W_s \in \mathcal{B(H)}$ is a unitary representation of $G$ on the same Hilbert space that satisfies  \begin{equation}\label{actioncond}W_s \mathcal{A}W_s^*=\mathcal{A}\end{equation} for all $s \in G$.  Given such a C$^*$-algebra, we can obtain a C$^*$-dynamical system $(\mathcal{A}, G, \hat{W})$ by defining the action $\hat{W}$ of $G$ on $\mathcal{A}$ by   
 \begin{equation}\label{actiondef}\hat{W}_s(A)=W_sAW_s^*\end{equation} for all $s\in G$ and  $A \in \mathcal{A}$.
   While the existence of a $*$-homomorphism from $\mathcal{A} \rtimes_{\hat{W}} G$ onto $ C^*(\mathcal{A}, \{W_s : s \in G\})$ is guaranteed (see, for example, \cite[Theorem 7.6.6] {Pedersen:1979}), the two C$^*$-algebras need not be isomorphic.    Several authors have determined conditions under which an isomorphism exists \cite{Karlovich:1988, Karlovich:2007, AntonevichLebedev:1994}.  Although the approaches of these authors differ slightly for noncommutative C$^*$-algebras, they coincide in the commutative case.  In the following, we restrict our attention to the setting of a commutative C$^*$-algebra $\mathcal{A}$.
  
  Let $M(\mathcal{A})$ be the collection of all non-trivial multiplicative linear functionals on $\mathcal{A}$.  The group $G$ is said to act topologically freely on $\mathcal{A}$ by the automorphisms $\hat{W}_s$ if, for every finite set $G_0 \subset G$ and every nonempty, open set $V \subset M(\mathcal{A})$, there exists $\pi \in V$ such that $\pi \circ \hat{W}_{s^{-1}} \neq \pi$ for all $s \in G_0 \setminus \{e\}$, where $e$ is the identity element of $G$.  We are using the weak-$*$ topology on $M(\mathcal{A}),$ but it is possible to use a weaker topology in this definition. 
    
  We can now state the following theorem that is the commutative version of Theorem 1 in \cite{Karlovich:1988} and Corollary 12.16 in \cite{AntonevichLebedev:1994}.
  
  \begin{theorem}\label{isothmcross} Let $\mathcal{A}$ be a  commutative, unital C$\, ^*$-algebra in $\mathcal{B(H)}$.  Let $G$ be a discrete, amenable group with 
  a unitary representation $s \mapsto W_s \in \mathcal{B(H)}$ that satisfies (\ref{actioncond}). Define the action $\hat{W}$ of $G$ on $\mathcal{A}$ by (\ref{actiondef}).  
  
  If $G$ acts topologically freely on $\mathcal{A}$ via the automorphisms $\hat{W}_s$, then $C^*(\mathcal{A}, \{W_s : s \in G\})$ is
   isometrically $*$-isomorphic to $\mathcal{A} \rtimes_{\hat{W}} G$.  The isomorphism $\Lambda : \mathcal{A} \rtimes_{\hat{W}} G \rightarrow C^*(\mathcal{A}, \{W_s : s \in G\})$ is
    defined on $C_c(G, \mathcal{A})$ by \begin{equation}\label{isoform}\Lambda\left(\sum_{s \in G} A_s\chi_s\right) = \sum_{s\in G} A_s W_s.\end{equation} \end{theorem}
    
  If $\mathcal{A}$ and $G$ satisfy the conditions above and $\mathcal{A}$ is a separable C$^*$-algebra, then the invertibility of operators in $C^*(\mathcal{A}, \{W_g : g \in G\})$ is often investigated using a method called the trajectorial approach.  For each multiplicative linear functional $\pi \in M(\mathcal{A})$, define a representation $\tau_{\pi}$ of $C^*(\mathcal{A}, \{W_s : s \in G\})$ on $\ell^2(G)$ by  \begin{equation}
  \label{howtodefinerep} (\tau_{\pi}(A)h)(s)=\pi(\hat{W}_s(A))h(s)   \qquad ( \tau_{\pi}(W_{s_0})h)(s)=h(s s_0). \end{equation}     The invertibility of $B \in C^*(\mathcal{A}, \{W_s : s \in G\})$ is then determined by the invertibility of its images under these representations in the following way:
  
  \begin{theorem} {\rm \cite[Theorem 21.2]{AntonevichLebedev:1994}} \label{trajectorial} Let $\mathcal{A}$, $G$, and $\hat{W}$ be defined as in Theorem \ref{isothmcross} and suppose that $\mathcal{A}$ is separable and $G$ acts topologically freely on $\mathcal{A}$ via the automorphisms $\hat{W_s}$.  Then an element $B \in C^*(\mathcal{A}, \{W_s : s \in G\})$ is invertible if and only if $\tau_{\pi}(B)$ is an invertible operator on $\ell^2(G)$ for all $\pi \in M(\mathcal{A})$.  \end{theorem}
  
  \section{Linear-fractionally-induced Composition Operators in $C^*(C_{\varphi}, \mathcal{K})$}\label{whatsinsection}
     
   To begin our investigation of the unital C$^*$-algebras generated by the compact operators and  composition operators induced by linear-fractional, non-automorphism self-maps $\varphi$ of $\ID$ that fix a point $\zeta \in \IT$, we determine which linear-fractionally-induced composition operators are contained in $C^*(C_{\varphi}, \mathcal{K}).$    Our arguments utilize two known lower bounds on the essential norm of a linear combination of composition operators on $H^2(\ID)$.
    
    If $\varphi$ is an analytic self-map of $\ID$, we set $J(\varphi):=\{\alpha \in \IT : |\varphi(\alpha)|=1\}$ and define $F(\varphi)$ to be the set of all points $\alpha \in \IT$ at which $\varphi$ has a finite angular derivative.  If $\varphi$ is a linear-fractional map, then $F(\varphi)=J(\varphi),$ and the finite angular derivative of $\varphi$ at $\alpha \in F(\varphi)$ is simply $\varphi^{\prime}(\alpha)$.  
  Using this notation, we recall the following bounds.  The first bound is a version of a result of Berkson \cite{Berkson:1981} and Shapiro and Sundberg \cite{ShapiroSundberg:1990} that appears in \cite{CowenMacCluer:1995} as Exercise 9.3.2. The second bound is Theorem 5.2 in \cite{KrieteMoorhouse:2007}. 
\begin{theorem} Let $\varphi_1, \ldots, \varphi_n$ be distinct analytic self-maps of $\ID$ and $c_1$, $\ldots$, $c_n \in \IC$.  Then
\begin{equation}\label{Jlowerbound} \left|\left|c_1C_{\varphi_1} + \ldots + c_n C_{\varphi_n}\right| \right|_e^2 \geq \frac{1}{2\pi} \sum_{j=1}^n|c_j|^2|J(\varphi_j)|, \end{equation}
  where $|J(\varphi_j)|$ denotes the Lebesgue measure of $J(\varphi_j)$.  In addition,  
  \begin{equation} \label{datalowerbound} \left|\left|c_{1}C_{\varphi_1} + \ldots + c_nC_{\varphi_n}\right|\right|_e^2 
\geq \sum_{ (d_0, d_1) \in \mathcal{D}_1(\zeta)} \left|\sum_{\substack{\zeta \in F(\varphi_j) \\ (\varphi_j(\zeta), \varphi_j^{\prime}(\zeta))= (d_0, d_1)}} c_j\right|^2\frac{1}{|d_1|},\end{equation} for all $\zeta \in \IT$, where $\mathcal{D}_{1}(\zeta):=\{(\varphi_j(\zeta), \varphi_j^{\prime}(\zeta)) : 1 \leq j \leq n, \zeta \in F(\varphi_j)\}$.
   \end{theorem}
   
   We now apply these bounds to determine the form of all composition operators contained in $C^*(C_{\varphi}, \mathcal{K})$.  The following lemma, and its proof, are modeled after Theorem 2 in \cite{KrieteMacCluerMoorhousePP}.
   
\begin{lemma}\label{whatthederivmustbe} Let $\varphi$ be a linear-fractional, non-automorphism self-map of $\ID$ with $\varphi(\zeta)=\zeta$ 
for some $\zeta \in \IT.$ Suppose $\psi: \ID \rightarrow \ID$ is analytic, $C_{\psi} \in C^{*}(C_{\varphi}, \mathcal{K}),$ and $F(\psi)$ is non-empty.  Then either 
$\psi(z)=z$ for all $z \in \ID$ or $F(\psi)=\{\zeta\}$, $\psi(\zeta)=\zeta$, and $\psi^{\prime}(\zeta)=\left(\varphi^{\prime}(\zeta)\right)^n$ for some 
$n \in \IZ.$ \end{lemma}

\begin{proof} Let $\mathcal{L}_{\varphi}$ be the collection of all maps $\ell$ that can be formed by composing copies of $\varphi$ and $\sigma_{\varphi}$.   All maps $\ell \in \mathcal{L}_{\varphi}$ are non-automorphism self-maps of $\ID$ that satisfy $\ell(\zeta)=\zeta$, 
 $J(\ell)=F(\ell)=\{\zeta\},$ and $\ell^{\prime}(\zeta)=\varphi^{\prime}(\zeta)^n$ for some $n \in \IZ.$  
  By (\ref{adjform}), $C_{\ell} \in C^*(C_{\varphi}, \mathcal{K})$ for all $\ell \in \mathcal{L}_{\varphi}$, and every word in $C_{\varphi}$ and $C_{\varphi}^*$ can be written as the sum of a compact operator and a constant multiple of a composition operator of this form.

Suppose  $\psi$ is not the identity map and $\psi \notin \mathcal{L}_{\varphi}$.  Let $\varepsilon >0$ be given.  Since $C_{\psi} \in C^*(C_{\varphi}, \mathcal{K})$, we can find
a constant $c_{\varepsilon} \in \IC$ and a linear combination $A_{\varepsilon}$ of composition 
operators induced by maps from $\mathcal{L}_{\varphi}$ such that $$\left|\left|C_{\psi}-A_{\varepsilon} - 
c_{\varepsilon}C_z\right|\right|_e  < \varepsilon.$$  Here, $C_z$ denotes the composition operator induced by the identity map on $\IC$,  which is the identity operator on $H^2(\ID)$.  Then, by (\ref{Jlowerbound}), $$\varepsilon^2 >
\left|\left|C_{\psi}-A_{\varepsilon} -c_{\varepsilon}C_z\right|\right|_e^2 \geq \frac{|J(\psi)|}{2\pi} + |c_{\varepsilon}|^2.$$ 
Thus, $|c_{\varepsilon}| < \varepsilon$, and $|J(\psi)|=0$   since $\varepsilon$ was arbitrary.  Hence
 \begin{equation}\label{bounding} \left|\left|C_{\psi}-A_{\varepsilon}\right|\right|_e \leq
  \left|\left|C_{\psi}- A_{\varepsilon} - c_{\varepsilon}C_z\right|\right|_e + \left|\left|c_{\varepsilon}C_z\right|\right|_e < 2 \varepsilon. \end{equation}

Suppose $\lambda \in F(\psi)$ with $\lambda \neq \zeta.$  Then, for all $\varepsilon >0$,
$$\left|\left|C_{\psi} - A_{\varepsilon}\right|\right|^2_e \geq |\psi^{\prime}(\lambda)|^{-1}$$
 by (\ref{datalowerbound}), which contradicts (\ref{bounding}) for sufficiently small $\varepsilon.$  Since $F(\psi)$ is assumed to be non-empty, 
 $F(\psi)=\{\zeta\}.$
 
 Similarly, if $\psi^{\prime}(\zeta) \neq \varphi^{\prime}(\zeta)^n$ for all $n \in \IZ$ or $\psi
 (\zeta) \neq \varphi(\zeta)$, then, for all $\varepsilon >0$,  \begin{equation}\label{tocontradict}\left|\left| C_{\psi}- A_{\varepsilon} \right| \right|_e^2  \geq |\psi^{\prime}(\zeta)|^{-1}\end{equation} 
 by (\ref{datalowerbound}) since $(\psi(\zeta), \psi^{\prime}(\zeta)) \neq (\ell(\zeta), \ell^{\prime}(\zeta))$ for all $\ell \in \mathcal{L}_{\varphi}.$
   For sufficiently small $\varepsilon,$ (\ref{tocontradict}) contradicts (\ref{bounding}), which proves that $\psi(\zeta)=\varphi(\zeta)$ and $\psi^{\prime}(\zeta)=\varphi^{\prime}(\zeta)^n$ for some $n \in \IZ$. \end{proof}

 We now restrict our attention to composition operators $C_{\psi}$ that are induced by linear-fractional maps.  We first consider the parabolic-induced composition operators.
 
  \begin{lemma}\label{gettheparabolics} If $\varphi$ is a linear-fractional, non-automorphism self-map of $\ID$ with $\varphi(\zeta)=\zeta$ for some
  $\zeta \in \IT$, then $C^*(\IP_{\zeta}) \subseteq C^*(C_{\varphi}, \mathcal{K}).$ \end{lemma}
 \begin{proof}   As noted in Theorem \ref{parabolicandourmaps}, $\sigma \circ \rho$ is a parabolic, non-automorphism self-map of $\ID.$  By (\ref{adjform}), there exists a compact operator $K$ such that $C_{\varphi}C_{\varphi}^* = (\varphi^{\prime}(\zeta))^{-1} C_{\varphi}C_{\sigma} +K$.  Hence $C_{\sigma \circ \varphi}  \in C^*(C_{\varphi}, \mathcal{K})$, and thus $C^*(\IP_{\zeta}) \subset C^*(C_{\varphi}, \mathcal{K})$ by Theorem \ref{paraboliciso}. \end{proof}
 
 Note that this lemma implies that, for all $\zeta \in \IT$, $C^*(\mathcal{F}_{\zeta})$ equals $C^*((\mathcal{F}_{\zeta} \setminus \IP_{\zeta}), \mathcal{K})$, the unital C$^*$-algebra generated by $\mathcal{K}$ and the collection of all composition operators induced by linear-fractional, non-automorphism self-maps $\varphi$ of $\ID$ that fix $\zeta$ and have $\varphi^{\prime}(\zeta) \neq 1.$

 We now fully characterize the non-compact, linear-fractionally-induced composition operators that are contained in $C^*(C_{\varphi}, \mathcal{K})$.  The proof follows the outline of the discussion preceding Theorem 7 in \cite{KrieteMacCluerMoorhousePP}.

 \begin{theorem}\label{whatsin} Let $\varphi$ be a linear-fractional, non-automorphism self-map of $\ID$ that fixes a point $\zeta \in \IT$, and suppose $\psi$ is a linear-fractional self-map of $\ID$ that satisfies $||\psi||_{\infty}=1$ and $C_{\psi} \neq I$. Then $C_{\psi} \in C^{*}(C_{\varphi}, \mathcal{K})$ if and only if $\psi(\zeta)=\zeta$, $\psi^{\prime}(\zeta)=\left(\varphi^{\prime}(\zeta)\right)^n$ for some $n \in \IZ$, and $\psi$ is not an automorphism of $
 \ID$. \end{theorem}
 \begin{proof}

 The forward direction is an immediate corollary of Lemma \ref{whatthederivmustbe}. For the reverse direction, suppose $\psi$ is a linear-fractional, non-automorphism self-map of $\ID$ that fixes $\zeta$ and satisfies
  $\psi^{\prime}(\zeta)= \varphi^{\prime}(\zeta)^n$ for some $n \in \IZ.$  If $n=0$, then
   $C_{\psi} \in C^*(C_{\varphi}, \mathcal{K})$ by Lemma \ref{gettheparabolics}. 
   
    If $n \in \IN$, 
   consider the map $\varphi_n$, the $n$th iterate of $\varphi$.  If $\psi(\ID) \subsetneq \varphi_n(\ID)$, then we define 
   $\rho:= \psi \circ \varphi_n^{-1},$ which is a parabolic, non-automorphism self-map of $\ID$ that
    fixes $\zeta$. Thus, there exists $a \in \IC$ with $\RE a >0$ such that $\rho=\rho_{\zeta, a}.$
Then $\psi = \rho_{\zeta, a} \circ \varphi_n,$ and  $C_{\psi} = C^n_{\varphi}C_{\rho_{\zeta, a}}
      \in C^*(C_{\varphi}, \mathcal{K}).$  If $\varphi_n (\ID) \subseteq \psi(\ID),$ define $\rho:= 
      \varphi_n \circ \psi^{-1}.$ The map $\rho$ is again a parabolic self-map of $\ID$, but, in this case, it can be an automorphism.  Thus, $ \rho = 
      \rho_{\zeta, \hat{a}}$ for some   $\hat{a} \in \IC$ with $\RE \hat{a} \geq 0$, and $\psi=\rho_{\zeta, \hat{a}}^{-1} \circ \varphi_n = \rho_{\zeta, 
      -\hat{a}} \circ \varphi_n.$   Since $\psi$ is a non-automorphism self-map of $\ID$, 
      $\RE \hat{a} < b_n$, where $b_n$ is the unique positive number satisfying
      $\rho_{\zeta, b_n}(\ID)=\varphi_n(\ID)$.  The existence of $b_n$ is guaranteed by Theorem \ref{parabolicandourmaps}.  Thus, regardless of the relationship between 
      $\psi(\ID)$ and $\varphi_n(\ID)$, we can write $\psi=\rho_{\zeta, c} \circ \varphi_n$ for
       some $c \in \IC$ with $\RE c > -b_n.$
       
       If $\RE a>0$, then, by applying  (\ref{adjform}), Theorem \ref{parabolicandourmaps}, and the properties of $*$-homomorphisms, we obtain that 
        \begin{align*} [(C_{\varphi_n}^*C_{\varphi_n})^a] &=[ C_{\varphi_n}^*C_{\varphi_n}]^a = \left(1/(\varphi^{\prime}(\zeta))^n\right)^a 
        [C_{\rho_{\zeta, 2b_n}}]^a = (1/(\varphi^{\prime}(\zeta))^{na})[C_{\rho_{\zeta, 2b_na}}].\end{align*} 
      By using the polar decomposition of $C_{\varphi_n}$ and relabeling $2b_na$ as $c$, we find that, for $\RE c >0$,
  
    \begin{align}\left[C_{\rho_{\zeta, c} \circ \varphi_n}\right] &= \left[C_{\varphi_n} C_{\rho_{\zeta, c}}\right]
       = \left[U_{\varphi_n}(C_{\varphi_n}^*C_{\varphi_n})^{1/2}C_{\rho_{\zeta, c}}\right]  \notag \\ &=\left[(\varphi^{\prime}(\zeta))^{nc/(2b_n)}U_{\varphi_n}(C_{\varphi_n}^*C_{\varphi_n})^{1/2+c/(2b_n)}\right].
      \label{equalincl} \end{align}    
      By the discussion in  \cite[Section 4.2]{KrieteMacCluerMoorhouse:2007}, $(\varphi^{\prime}(\zeta))^{nc/(2b_n)}U_{\varphi_n}(C_{\varphi_n}^*C_{\varphi_n})^{1/2+c/(2b_n)} \in C^*(C_{\varphi_n}, \mathcal{K})$ for all $c \in \IC$ with $\RE c > -b_n$.  Applying the arguments used in \cite{KrieteMacCluerMoorhousePP}, one can easily show that 
 $\left[(\varphi^{\prime}(\zeta))^{nc/(2b_n)}U_{\varphi_n}(C_{\varphi_n}^*C_{\varphi_n})^{1/2+c/(2b_n)}\right]$  and $[C_{\rho_{\zeta, c} \circ \varphi_n}]$ are both holomorphic functions of $c$ on $\{c\in \IC: \RE c >-b_n\}.$ 
Thus, (\ref{equalincl}) holds for all $c \in \IC$ with $\RE c > -b_n$, which proves that $C_{\psi} \in C^*(C_{\varphi}, \mathcal{K}).$

    The proof for $n <0$ is similar with $\sigma$ taking the place of $\varphi$ in the arguments.
 \end{proof}
 
\section{The Structures of $C^*(\mathcal{F}_{\zeta})$, $C^*(C_{\varphi}, \mathcal{K})$, and $C^*(C_{\varphi_1}, \ldots, C_{\varphi_n}, \mathcal{K})$ Modulo the Ideal of Compact Operators}\label{structuresection}

   For $\zeta \in \IT$ and $t>0$, we define the automorphism $\Psi_{\zeta, t}$ of $\ID$ by \begin{equation}\label{psit} \Psi_{\zeta, t}(z)= \frac{(t+1)z + (1-t)\zeta}{(1-t)\overline{\zeta}z + (1+t)}. \end{equation}   Note that $\Psi_{\zeta, t}(\zeta)=\zeta$, $\Psi^{\prime}_{\zeta, t}(\zeta)=t$,  $\Psi^{\prime \prime}_{\zeta, t}(\zeta)=(t^2-t)\overline{\zeta}$, and $\Psi^{-1}_{\zeta, t}(0)=\zeta(t-1)/(t+1)$.  Straightforward calculations show that, for all $\zeta \in \IT$ and $t_1, t_2 >0$, \begin{equation}\label{psicommute}\Psi_{\zeta, t_1}\circ \Psi_{\zeta, t_2} = \Psi_{\zeta, t_1t_2},\end{equation}  so $\{\Psi_{\zeta, t} : t>0\}$ is an abelian group of automorphisms.  The relationship between the composition operators induced by these automorphisms and the composition operators in $\mathcal{F}_{\zeta}$ is described by the following lemma and its proof.
   
   \begin{lemma}\label{formoffullfixset} Let $\zeta \in \IT$.  Then $\mathcal{F}_{\zeta} = \{C_{\rho_{\zeta, a}}C_{\Psi_{\zeta, t}} : a \in \Upsilon, t>0\}$.  \end{lemma}
   \begin{proof}
      If $a \in \Upsilon$ and $t>0$, then $\Psi_{\zeta, t} \circ \rho_{\zeta, a}$ is a linear-fractional, non-automorphism self-map of $\ID$ that fixes $\zeta$.  Thus, $C_{\rho_{\zeta, a}}C_{\Psi_{\zeta, t}} \in \mathcal{F}_{\zeta}$ for all $a \in \Upsilon$ and $t>0$.  Note that $(\Psi_{\zeta, t} \circ \rho_{\zeta, a})^{\prime}(\zeta)=t$. 

        If $\varphi$ is a linear-fractional, non-automorphism self-map of $\ID$ that fixes $\zeta$, then $\varphi^{\prime}(\zeta) >0$ and $a_{\varphi}:=(\varphi^{\prime \prime}(\zeta)\zeta - \varphi^{\prime}(\zeta)^2+\varphi^{\prime}(\zeta))/\varphi^{\prime}(\zeta)$ has positive real part.  This statement is clear for affine maps.  
      For non-affine maps, it follows from the fact that every non-affine, linear-fractional self-map of $\ID$ that fixes $\zeta$ and is not an automorphism of $\ID$ has the form $$\varphi(z)=\frac{(1+\varphi^{\prime}(\zeta) +
       \varphi^{\prime}(\zeta)d)z + (d- \varphi^{\prime}(\zeta)- \varphi^{\prime}(\zeta) d) \zeta}{\overline{\zeta}z+d}$$ for some $d \in \IC$ with $\RE \left[(d-1)/(d+1)\right]>\varphi^{\prime}(\zeta)$ \cite{BasorRetsek:2006}.   It is easy to show that $\varphi$ and $\Psi_{\zeta, \varphi^{\prime}(\zeta)} \circ \rho_{\zeta, a_\varphi}$  have the same first and second derivatives at $\zeta$.  Since a linear-fractional map $\psi$ is fully determined by the values of $\psi(z_0)$, $\psi^{\prime}(z_0)$, and $\psi^{\prime \prime}(z_0)$ for any $z_0 \in\IC$ with $\psi(z_0) \neq \infty$, we obtain that $\varphi=\Psi_{\zeta, \varphi^{\prime}(\zeta)} \circ \rho_{\zeta, a_{\varphi}}$ and $C_{\varphi} \in \{C_{\rho_{\zeta, a}}C_{\Psi_{\zeta, \varphi^{\prime}(\zeta)^n}} : a \in \Upsilon, n \in \IZ\} \subseteq \{C_{\rho_{\zeta, a}}C_{\Psi_{\zeta, t}} : a \in \Upsilon, t >0\}$.
       \end{proof}

  If $a \in \Upsilon$ and $t>0$, then, by Theorems \ref{comprel} and \ref{unitaryrel}, there exist compact operators $K_1$ and $K_2$ such that \begin{align}\label{howtowritewithu} C_{\rho_{\zeta, a}}C_{\Psi_{\zeta, t}}&= C_{\rho_{\zeta, a}}T_{\frac{|(t+1)-\overline{\zeta}(t-1)z|}{2\sqrt{t}}}U_{\Psi_{\zeta,t }} + K_1 =t^{-1/2}C_{\rho_{\zeta, a}}U_{\Psi_{\zeta, t}} +K_2.\end{align}
   Thus, for all $\zeta \in \IT$, \begin{align}\label{formfullfixed}C^*\left(\mathcal{F}_{\zeta}\right) = C^*\left(\left\{C_{\rho_{\zeta, a}}U_{\Psi_{\zeta, t}} : a \in \Upsilon, t>0\right\}\right).  \end{align}
If $\varphi$ is a linear-fractional, non-automorphism self-map of $\ID$ that fixes $\zeta \in \IT$, then by Theorem \ref{whatsin} and the proof of Lemma \ref{formoffullfixset},  \begin{align}\label{formsingle}C^*\left(C_{\varphi}, \mathcal{K}\right)=C^*\left(\left\{C_{\rho_{\zeta, a}}U_{\Psi_{\zeta,\varphi^{\prime}(\zeta)^n}} : a \in \Upsilon, n \in \IZ\right\}\right).\end{align}
Similarly, if $n \in \IN$ and $\varphi_{1}, \ldots, \varphi_n$ are linear-fractional, non-automorphism self-maps of $\ID$ that fix a given point $\zeta \in \IT$, then \begin{align}\label{formlinind}& C^*(C_{\varphi_1}, \ldots, C_{\varphi_n}, \mathcal{K}) \\ & \quad = C^*\left(\left\{C_{\rho_{\zeta, a}} U_{\Psi_{\zeta, \varphi_1^{\prime}(\zeta)^{m_1} \ldots \varphi_n^{\prime}(\zeta)^{m_n}}} : a \in \Upsilon, (m_1, \ldots, m_n) \in \IZ^n\right\}\right). \notag \end{align}

The C$^*$-algebras on the right-hand sides of  (\ref{formfullfixed}), (\ref{formsingle}), and (\ref{formlinind})  are  subalgebras of C$^*$-algebras of the form $C^*(\IP_{\zeta}, \{U_{\gamma} : \gamma \in G\}),$ where $G$ is an abelian group of automorphisms of $\ID$ that fix the point $\zeta$.
  In (\ref{formfullfixed}), $G=\{\Psi_{\zeta, t} : t>0\},$ which is isomorphic to $\IR^+$, the multiplicative group of positive real numbers.    
In the setting  of (\ref{formsingle}), $G=\{\Psi_{\zeta, \varphi^{\prime}(\zeta)^n} : n \in \IZ\},$ which is isomorphic to $\IZ$ if $\varphi^{\prime}(\zeta) \neq 1.$  
For (\ref{formlinind}), $G=\{\Psi_{\zeta, \varphi_1^{\prime}(\zeta)^{m_1} \ldots \varphi_n^{\prime}(\zeta)^{m_n}} : (m_1, \ldots, m_n) \in \IZ^n\}.$
   This group is isomorphic to $\IZ^n$ if $\ln(\varphi_1^{\prime}(\zeta)), \ldots, \ln(\varphi_n^{\prime}(\zeta))$ are linearly independent over $\IZ$.

    \subsection{The Structure of $C^*(\IP_{\zeta}, \{U_{\gamma} : \gamma \in G\})/\mathcal{K}$}
    Motivated by the preceding discussion, we will determine the structure, modulo the ideal of compact operators, of the unital C$^*$-algebra generated by $\IP_{\zeta}$ and $\{U_{\gamma} :  \gamma \in G\}$, where $G$ is an abelian group  of automorphisms of $\ID$ that satisfy $\gamma(\zeta)=\zeta$ for a given $\zeta \in \IT$. We will show that this C$^*$-algebra satisfies the hypotheses of Theorem  \ref{isothmcross}   and is therefore a crossed product C$^*$-algebra.  The strategy employed in our arguments is similar to the one used by Jury in his investigation of the structure of $C^*(T_z, \{U_{\gamma} : \gamma \in G\})$ \cite{Jury:2007G}.  The main difference is that the unit circle is replaced by the closed unit interval.  The automorphisms in $G$ are not, in general, self-maps of the unit interval, so $G$ does not act on $C([0,1])$ via composition.  Hence the  action of  $G$ that appears in the crossed product  is more complicated in this setting.
    
   In the following, $\id$ denotes the identity map on $\IC$, and $G_d$ is the group $G$ with the discrete topology.
  
  \begin{theorem}\label{p1u}  Let $\zeta \in \IT$, and let $G$ be a collection of automorphisms of $\ID$ that fix $\zeta$.  Suppose that $G$ is an abelian group under composition and that $\gamma^{\prime}(\zeta)\neq 1$ for all $\gamma \in G \setminus \{\id\}$.  Define the action $\beta: G_d \rightarrow {\rm Aut}(C([0,1]))$ by $\beta_{\gamma}(f)(x):=f\left(x^{\gamma^{\prime}(\zeta)}\right)$ for all $\gamma \in G$, $f \in C([0,1])$, and $x \in [0,1]$.  Then there exists a $*$-homomorphism $\omega : C^*(\IP_{\zeta}, \{U_{\gamma} : \gamma \in G\}) \rightarrow C([0,1]) \rtimes_{\beta} G_d$ such that  $$0 \rightarrow \mathcal{K} \hookrightarrow C^*(\IP_{\zeta}, \{U_{\gamma} : \gamma \in G\}) \stackrel{\omega}{\rightarrow} C([0,1]) \rtimes_{\beta} G_d \rightarrow 0$$ is a short exact sequence. \end{theorem}
  
  \begin{proof}  We first note that $C^*(\IP_{\zeta}, \{U_{\gamma}: \gamma \in G\})/ \mathcal{K}$ is the C$^*$-algebra generated by $C^*(\IP_{\zeta})/\mathcal{K}$ and $\{[U_{\gamma}] : \gamma \in G_d\}.$  The map $\gamma \mapsto [U_{\gamma}]$ is a group homomorphism of $G_d$ onto a group of unitary elements by Theorem \ref{unitaryrel} and the assumption that $G_d$ is abelian.  We want to apply Theorem \ref{isothmcross} to this C$^*$-algebra, so we need to show that \begin{equation}\label{neededcrossrel} [U_{\gamma}]\left(C^*(\IP_{\zeta})/\mathcal{K}\right)[U_{\gamma}^*] = C^*(\IP_{\zeta})/\mathcal{K} \end{equation} for all $\gamma \in G_d$.
  
  For each $\gamma \in G_d$, the map $\beta_{\gamma}$ is a $*$-automorphism of $C([0,1])$ since $\gamma^{\prime}(\zeta) >0$.  It is straightforward to verify that $\beta$ is an action of $G_d$ on $C([0,1])$.  Let $\Gamma_{\zeta}: C([0,1]) \rightarrow C^*(\IP_{\zeta})/\mathcal{K}$ be the $*$-isomorphism from Theorem  \ref{paraboliciso}.  We want to show that \begin{align}\label{covariance} [U_{\gamma}]\Gamma_{\zeta}(f)[U_{\gamma}^*]=\Gamma_{\zeta}(\beta_{\gamma}(f))\end{align} for all $\gamma \in G_d$ and $f \in C([0,1])$.  
  Since $\beta_\gamma$ is a $*$-automorphism of $C([0,1])$ and $\Gamma_{\zeta}$ is a $*$-isomorphism, this will prove that (\ref{neededcrossrel}) holds for all $\gamma \in G_d$.  Moreover, (\ref{covariance}) implies that, if we define the action $\hat{W}$ of $G_d$ on $C^*(\IP_{\zeta})/\mathcal{K}$ by $\hat{W}_{\gamma}(A)=[U_{\gamma}]A[U_{\gamma}^*]$ for all $\gamma \in G_d$ and $A \in C^*(\IP_{\zeta})/\mathcal{K}$, then \begin{align}\label{relbtactions} \hat{W_{\gamma}} \circ \Gamma_{\zeta} = \Gamma_{\zeta} \circ \beta_\gamma \end{align} for all $\gamma \in G_d.$
  
  We begin by considering functions of the form $x^a$ for $a \in \Upsilon \cup \{0\}$.  For all $\gamma \in G_d$, \begin{align*} [U_{\gamma}]\Gamma_{\zeta} \left(x^a\right)[U_{\gamma}^*]  &= \left[U_{\gamma}C_{\rho_{\zeta, a}}U_{\gamma}^*\right] 
= \left[T_{w_{\gamma}}C_{\gamma}C_{\rho_{\zeta, a}} U_{\gamma}^*\right] \displaybreak[1]\\
& = \left[T_{w_{\gamma}}C_{\rho_{\zeta, \gamma^{\prime}(\zeta)a}}C_{\gamma} U_{\gamma}^*\right] 
 = \left[C_{\rho_{\zeta, \gamma^{\prime}(\zeta)a}}T_{w_{\gamma}}C_{\gamma}U_{\gamma}^*\right]  \label{cov3} \\
& = \Gamma_{\zeta}\left(x^{a \gamma^{\prime}(\zeta)}\right)[U_{\gamma}U_{\gamma}^*] = \Gamma_{\zeta}\left(\beta_{\gamma}^G\left(x^{a}\right)\right).\notag
\end{align*}
The second and fifth equalities follow from Theorem \ref{unitaryrel}.  The third equality is by Theorem \ref{parabolicandourmaps}, and the fourth equality is a consequence of  Theorem \ref{comprel}.  Extending by linearity, we see that  (\ref{covariance}) holds for all $\gamma \in G_d$ and a dense collection of functions $f$ in $C([0,1])$.  Therefore, the relationship is true for all $f \in C([0,1]).$  

Every non-trivial multiplicative linear functional on $C^*(\IP_{\zeta})/\mathcal{K}$ has the form $ev_x \circ \Gamma^{-1}_{\zeta}$ for some $x \in [0,1]$, where $ev_x$ denotes the linear functional on $C([0,1])$ of evaluation at $x$.  If $\gamma \in G_d$ and $x \in [0,1]$, then, by (\ref{relbtactions}), $$(ev_x \circ \Gamma_{\zeta}^{-1}) \circ \hat{W}_{\gamma^{-1}} = ev_x \circ \beta_{\gamma^{-1}} \circ \Gamma_{\zeta}^{-1} = ev_{x^{(\gamma^{\prime}(\zeta))^{-1}}} \circ \Gamma_{\zeta}^{-1}.$$ 
If $x \in (0,1)$ and $\gamma \in G_d \setminus \{\id\}$,  then $x^{\gamma^{\prime}(\zeta)^{-1}} \neq x$ and $(ev_x \circ \Gamma_{\zeta}^{-1}) \circ \hat{W}_{\gamma^{-1}} \neq ev_x \circ \Gamma_{\zeta}^{-1}$ since $\gamma^{\prime}(\zeta) \neq 1$.   Thus, $G_d$ acts topologically freely on $C^*(\IP_{\zeta})/\mathcal{K}$ by the automorphisms $\hat{W}_{\gamma}$.  Hence, by Theorem \ref{isothmcross}, $C^*(C^*(\IP_{\zeta})/\mathcal{K}, \{[U_{\gamma}] : \gamma \in G_d\})$ is isometrically $*$-isomorphic to $C^*(\IP_{\zeta})/\mathcal{K} \rtimes_{\hat{W}} G_d.$  Since $\Gamma_{\zeta}$ is a $*$-isomorphism between $C([0,1])$ and $C^*(\IP_{\zeta})/\mathcal{K}$ and $\Gamma_{\zeta}$, $\hat{W}$, and $\beta$ satisfy (\ref{relbtactions}), $C^*(\IP_{\zeta})/\mathcal{K} \rtimes_{\hat{W}} G_d$ is isometrically $*$-isomorphic to $C([0,1]) \rtimes_{\beta} G_d$ by Lemma 2.65 in \cite{Williams:2007}.\end{proof}

If $\gamma$ is an automorphism of $\ID$ that fixes $\zeta \in \IT$, let $\gamma_n$ denote the $n$th iterate of $\gamma$ if $n >0$, the $|n|$th iterate of $\gamma^{-1}$ if $n < 0$, and the identity map if $n=0$.  Then $C^*(\IP_{\zeta}, U_{\gamma})/\mathcal{K} = C^*(\IP_{\zeta}, \{U_{\gamma_n} : n \in \IZ\})/\mathcal{K}$ and $\gamma_n^{\prime}(\zeta)=(\gamma^{\prime}(\zeta))^n$ for all $n \in \IZ$.  Thus, the following corollary is an immediate consequence of the preceding theorem.

\begin{corollary}\label{p2u} Suppose $\gamma$ is an automorphism of $\ID$ that fixes a point $\zeta \in \IT$ and satisfies $\gamma^{\prime}(\zeta) \neq 1$.  Define the action $\beta^{\gamma} : \IZ \rightarrow {\rm Aut}(C([0,1]))$ by $\beta^{\gamma}_n(f)(x):=f\left(x^{\gamma^{\prime}(\zeta)^n}\right)$ for all $n \in \IZ,$ $f \in C([0,1])$, and $x \in [0,1]$.  Then there exists a $*$-homomorphism $\omega^{\gamma}: C^*(\IP_{\zeta}, U_{\gamma}) \rightarrow C([0,1]) \rtimes_{\beta^{\gamma}} \IZ$ such that  $$ 0 \rightarrow \mathcal{K} \hookrightarrow C^*(\IP_{\zeta}, U_{\gamma}) \stackrel{\omega^{\gamma}}{\rightarrow} C([0,1]) \rtimes_{\beta^{\gamma}} \IZ \rightarrow 0$$ is a short exact sequence. \end{corollary}

We now consider C$^*$-subalgebras of $C^*(\IP_{\zeta}, \{U_{\gamma} : \gamma \in G\})$ of the form described in (\ref{formfullfixed})--(\ref{formlinind}).

\begin{theorem}\label{p1usub}  Let $\zeta \in \IT$, and let $G$ be a collection of automorphisms of $\ID$ that fix $\zeta$.  Suppose $G$ is an abelian group under composition and that $\gamma^{\prime}(\zeta)\neq 1$ for all $\gamma \in G \setminus \{{\id}\}.$  Then the C$^{\, *}$-algebra $C^*(\{[C_{\rho_{\zeta, a}}U_{\gamma}] : a \in \Upsilon, \gamma \in G\})$ is isometrically $*$-isomorphic to the unitization of $C_0([0,1]) \rtimes_{\beta} G_d$, where the action $\beta: G_{d} \rightarrow {\rm Aut}(C_0([0,1]))$ is defined by $\beta_{\gamma} (f)(x):=f\left(x^{\gamma^{\prime}(\zeta)}\right)$ for all $\gamma \in G_d,$ $f \in C_{0}([0,1]),$ and $x \in [0,1]$. \end{theorem}

\begin{proof}  
The action $\beta$ from Theorem \ref{p1u} satisfies $\beta_{\gamma}(C_0([0,1]))=C_0([0,1])$ for all $\gamma \in G$.  Since $C_0([0,1])$ is an ideal in $C([0,1])$, each $\beta_{\gamma}$ restricts to an automorphism of $C_0([0,1])$, which is again called $\beta_{\gamma}$, and $\beta$ is an action on $C_0([0,1])$.

Let $N_G:=\{\Gamma_{\zeta}(f)[U_{\gamma}] : f\in C_0([0,1]), \gamma \in G\}$.  For all $a \in \Upsilon$ and $\gamma \in G$, $\Gamma_{\zeta}(x^a)[U_{\gamma}]
=[C_{\rho_{\zeta, a}}U_{\gamma}] \in C^*(\{[C_{\rho_{\zeta, a}}U_{\gamma}] : a \in \Upsilon, \gamma \in G\})$. 
 Since finite linear combinations of the functions $\{x^a : a \in \Upsilon\}$ are dense in $C_0([0,1])$, $C^*(N_G)=C^*(\{[C_{\rho_{\zeta, a}}U_{\gamma}] : a \in \Upsilon, \gamma \in G\}).$ 

We want to show that the set $N_G$ is closed under adjoints and products.  Suppose $f_1, f_2 \in C_0([0,1])$ and $\gamma_1, \gamma_2 \in G$.  Then, by (\ref{covariance}) and Theorem \ref{unitaryrel}, \begin{equation*} \left(\Gamma_{\zeta}(f_1)\left[U_{\gamma_1}\right]\right)^* = \left[U_{\gamma_1^{-1}}\right]\Gamma_{\zeta}\left(\overline{f_1}\right) = \Gamma_{\zeta}\left(\beta_{\gamma_1^{-1}}\left(\overline{f_1}\right)\right)\left[U_{\gamma_1^{-1}}\right]  \in N_G\end{equation*} and 
\begin{align*} \Gamma_{\zeta}(f_1)\left[U_{\gamma_1}\right] \Gamma_{\zeta}(f_2) \left[U_{\gamma_2}\right] &= \Gamma_{\zeta}(f_1)\Gamma_{\zeta}\left(\beta_{\gamma_1}(f_2)\right) \left[U_{\gamma_1}\right]\left[U_{\gamma_2}\right]\\ &= \Gamma_{\zeta}\left(f_1\beta_{\gamma_1}(f_2)\right)\left[U_{\gamma_2 \circ \gamma_1}\right] \in N_G.\end{align*}  Thus, if we define $N_G^{\prime}$ to be the set of all finite linear combinations of cosets in $N_G$, then $N^{\prime}_G$  is a $*$-subalgebra of $C^*(\IP_\gamma, \{U_{\gamma} : \gamma \in G\})/ \mathcal{K}$, and $\IC[I] + N_{G}^{\prime}$ is dense in $C^*(N_G)$.

Let $\Lambda$ be the $*$-isomorphism from $C([0,1]) \rtimes_{\beta} G_d$ onto $C^*(\IP_{\gamma}, \{U_{\gamma} : \gamma \in G\})/\mathcal{K}$ of form (\ref{isoform}) that was obtained in the proof of Theorem \ref{p1u}.  For $c \in \IC$ and $\sum_{\gamma \in G_0} \Gamma_{\zeta}(f_{\gamma})[U_{\gamma}] \in N^{\prime}_G$, $$\left(\Lambda\right)^{ -1}\left(c[I] + \sum_{\gamma \in G_0} \Gamma_{\zeta}(f_{\gamma})[U_{\gamma}]\right) = c\chi_{\id} + \sum_{\gamma \in G_0} f_{\gamma}\chi_{\gamma}.$$  Thus, $\left(\Lambda\right)^{ -1}$ maps $\IC[I] + N_{G}^{\prime}$ onto $\IC \chi_{\id} + C_c(G_d, C_0([0,1])).$  Hence $C^*(N_G)$ is isomorphic to the closure of $\IC \chi_{\id} + C_c(G_d, C_0([0,1]))$ in $C([0,1]) \rtimes_{\beta} G_d$.

  By \cite[Lemma 3.17]{Williams:2007}, the closure of $C_c(G_d, C_0([0,1]))$ in $C([0,1]) \rtimes_{\beta} G_d$ is $*$-isomorphic to $C_0([0,1]) \rtimes_{\beta} G_d$, which is a non-unital C$^*$-algebra since $C_0([0,1])$ is non-unital (see \cite[II.10.3.9]{Blackadar:2006}).  Therefore, the closure of $\IC \chi_{\id} + C_c(G_d, C_0([0,1]))$ is isometrically $*$-isomorphic to the unitization of $C_0([0,1]) \rtimes_{\beta} G_d$. \end{proof}

\subsection{Structure Results for $C^*(\mathcal{F}_{\zeta})$, $C^*(C_{\varphi}, \mathcal{K}),$ and $C^*(C_{\varphi_1}, \ldots, C_{\varphi_n}, \mathcal{K})$}
We  now apply Theorem \ref{p1usub}  to determine the structures of $C^*(\mathcal{F}_{\zeta})$, $C^*(C_{\varphi}, \mathcal{K})$, and $C^*(C_{\varphi_1}, \ldots, C_{\varphi_n}, \mathcal{K})$, modulo the ideal of compact operators, for all $\zeta \in \IT$ and $C_{\varphi}, C_{\varphi_1}, \ldots, C_{\varphi_n} \in \mathcal{F}_{\zeta}$.   The following results are immediate consequences of Theorem \ref{p1usub} and the discussion preceding Theorem \ref{p1u}.

\begin{theorem} Let $\zeta \in \IT$.  Define the action $\beta: \IR^+_d \rightarrow {\rm Aut}(C_0([0,1]))$ by $\beta_t(f)(x):=f\left(x^t\right)$ for all $f \in C_0([0,1])$, $t \in \IR^+_d$,  and $x \in [0,1]$.  Then the C$^{\, *}$-algebra $C^*(\mathcal{F}_{\zeta})/\mathcal{K}$ is isometrically $*$-isomorphic to the unitization of $C_0([0,1]) \rtimes_{\beta} \IR^+_d$. \end{theorem}

\begin{theorem}\label{structureonephi} Let $\zeta \in \IT$, and suppose $\varphi$ is a linear-fractional, non-automorphism self-map of $\ID$ that satisfies  $\varphi(\zeta)=\zeta$ and $\varphi^{\prime}(\zeta)\neq 1.$  Define the action $\beta^{\varphi}: \IZ \rightarrow {\rm Aut}(C_0([0,1]))$ by $\beta^{\varphi}_{n}(f)(x):=f\left(x^{\varphi^{\prime}(\zeta)^n}\right)$ for all $f \in C_0([0,1])$, $n \in \IZ$,  and $x \in [0,1]$.  Then the C$^{\, *}$-algebra $C^*(C_{\varphi}, \mathcal{K})/\mathcal{K}$ is isometrically $*$-isomorphic to the unitization of $C_{0}([0,1]) \rtimes_{\beta^\varphi} \IZ.$ \end{theorem}

\begin{theorem}\label{structurelinind} Let $\zeta \in \IT$ and $n \in \IN$.  Suppose $\varphi_1, \ldots, \varphi_n$ are linear-fractional, non-automorphism self-maps of $\ID$ that fix $\zeta$. Define the action $\beta^{\prime}: \IZ^n \rightarrow {\rm Aut}(C_0([0,1]))$ by $\beta^{\prime}_{(m_1, \ldots, m_n)}(f)(x):=
f\left(x^{\varphi_1^{\prime}(\zeta)^{m_1} \ldots \varphi_n^{\prime}(\zeta)^{m_n}}\right)$ for all $f \in C_0([0,1])$, $(m_1, \ldots, m_n) \in \IZ^n$,  and $x \in [0,1]$.
  If  $\ln(\varphi_1^{\prime}(\zeta)), \ldots, \ln(\varphi_n^{\prime}(\zeta))$ are linearly independent over $\IZ$, then the C$^{\, *}$-algebra
   $C^*(C_{\varphi_1}, \ldots, C_{\varphi_n} , \mathcal{K})/\mathcal{K}$ is isometrically $*$-isomorphic to the unitization of $C_0([0,1]) \rtimes_{\beta^{\prime}} \IZ^n.$ \end{theorem}
   
   Although Theorem \ref{structureonephi} is a special case of Theorem \ref{structurelinind}, we list it separately due to the importance of the single operator case.

\section{K-theory and Spectral Results}\label{applicationsection}
In this section, we determine the K-theory of $C^*(C_{\varphi}, \mathcal{K})$ and investigate the essential spectra of operators in this C$^*$-algebra.  These two parts are connected because the K-theory results determine the possible values of the Fredholm index for  Fredholm operators in $C^*(C_{\varphi}, \mathcal{K})$.

\subsection{K-theory}  
To compute the K-groups of $C^*(C_{\varphi}, \mathcal{K})$, we apply two cyclic, six-term exact sequences: the standard six-term exact sequence of K-groups that corresponds to any given short exact sequence of C$^*$-algebras (see, for example, \cite[Theorem 12.1.2]{Rordam:2000}) and the Pimsner-Voiculescu exact sequence for crossed products by $\IZ$  \cite{PimsnerVoiculescu:1980}.  For more information about K-theory, we direct the reader to \cite{Rordam:2000, Blackadar:1998}.

     \begin{theorem}\label{onephiKtheory} Let $\zeta \in \IT$, and let $\varphi$ be a linear-fractional, non-automorphism self-map of $\ID$ that satisfies $\varphi(\zeta) = \zeta$ and $\varphi^{\prime}(\zeta) \neq 1.$  Then $K_0(C^*(C_{\varphi}, \mathcal{K})) \cong \IZ \oplus \IZ$ and $K_1(C^*(C_{\varphi}, \mathcal{K})) \cong 0$. Moreover, $K_0(C^*(C_{\varphi}, \mathcal{K}))$ is generated by $[I]_0$ and $[F]_0$, where $F$ is an arbitrary one-dimensional projection in $\mathcal{K}$.  \end{theorem}
    
    \begin{proof} 
    Let $\beta^{\varphi}: \IZ \rightarrow {\rm Aut}(C_0([0,1]))$ be the action defined in Theorem \ref{structureonephi}.  Since  $K_0(C_0([0,1])) \cong K_1(C_0([0,1])) \cong 0,$ 
 the Pimsner-Voiculescu exact sequence for the crossed product $C_0([0,1]) \rtimes_{\beta^{\varphi}} \IZ$ has the form
\[ \xymatrix{0 \ar[r] & 0 \ar[r] & K_0(C_0([0,1]) \rtimes_{\beta^{\varphi}} \IZ) \ar[d] \\
K_1(C_0([0,1]) \rtimes_{\beta^{\varphi}} \IZ) \ar[u] & 0 \ar[l] & 0 \ar[l]}.\] Thus, $K_0(C_0([0,1]) \rtimes_{\beta^{\varphi}} \IZ) \cong K_1(C_0([0,1]) \rtimes_{\beta^{\varphi}} \IZ) \cong 0$.

Let $\mathcal{C}$ denote the unitization of $C_0([0,1]) \rtimes_{\beta^{\varphi}} \IZ$.  By  \cite[Example 4.3.5 and Proposition 8.1.6]{Rordam:2000}, $K_0(\mathcal{C}) \cong K_0(C_0([0,1]) \rtimes_{\beta^{\varphi}} \IZ) \oplus \IZ \cong \IZ$ and $K_1(\mathcal{C}) \cong K_1(C_0([0,1]) \rtimes_{\beta^{\varphi}} \IZ) \cong 0$.   Straightforward calculations show that $K_0(\mathcal{C})$ is generated by $[I_{\mathcal{C}}]_0$.
By Theorem \ref{structureonephi}, there exists a short exact sequence $$0 \longrightarrow \mathcal{K} \longrightarrow C^*(C_{\varphi}, \mathcal{K}) \stackrel{\omega}{\longrightarrow} \mathcal{C} \longrightarrow 0.$$  
Since $K_0(\mathcal{K}) \cong \IZ$ and $K_1(\mathcal{K}) \cong 0,$ 
the six-term exact sequence corresponding to  this short exact sequence is \[\xymatrix{ \IZ \ar[r] & K_0(C^*(C_{\varphi}, \mathcal{K})) \ar[r] & \IZ \ar[d] \\ 0 \ar[u] & K_1(C^*(C_{\varphi}, \mathcal{K})) \ar[l] & 0 \ar[l] }\]  
All K-groups are abelian, and $\IZ$ is a free abelian group. Hence, $K_0(C^*(C_{\varphi}, \mathcal{K})) \cong \IZ \oplus \IZ$, and $K_1(C^*(C_{\varphi}, \mathcal{K})) \cong 0$ (see, for example, \cite[Corollary 1.3.38]{Maunder:1980}).  The proof of Theorem \ref{structureonephi} shows that $\omega(I)=I_{\mathcal{C}}$.  Since $[I_{\mathcal{C}}]_0$ generates $K_0(\mathcal{C})$ and $K_0(\mathcal{K})$ is generated by $[F]_0$, for any one-dimensional projection $F$ in $\mathcal{K}$, the generators of $K_0(C^*(C_{\varphi}, \mathcal{K}))$ are $[I]_0$ and $[F]_0$.
\end{proof}

By combining  the fact that $K_1(C^*(C_{\varphi}, \mathcal{K})/\mathcal{K}) \cong 0$ with the connection between the Fredholm index and the K-theory index map \cite[Proposition 9.4.2]{Rordam:2000},  we obtain the following result about the index of Fredholm operators in $C^*(C_{\varphi}, \mathcal{K}).$

\begin{corollary} Let $\zeta \in \IT$, and let $\varphi$ be a linear-fractional, non-automorphism self-map of $\ID$ that satisfies $\varphi(\zeta) = \zeta$ and $\varphi^{\prime}(\zeta) \neq 1.$   Then the Fredholm index of every Fredholm operator in $C^*(C_{\varphi}, \mathcal{K})$ is $0.$ \end{corollary}


\subsection{Spectral Results for Operators in $C^*(\IP_{\zeta}, U_{\gamma})$}
Recall that if $\varphi$ is a linear-fractional, non-automorphism self-map of $\ID$ that fixes $\zeta \in \IT$ and satisfies $\varphi^{\prime}(\zeta) \neq 1,$ then $C^*(C_{\varphi}, \mathcal{K})/\mathcal{K}$ is a subalgebra of a C$^*$-algebra of the form 
 $C^*(\IP_{\zeta}, U_{\gamma})/\mathcal{K},$ where $\gamma$ is an automorphism of $\ID$ such that $\gamma(\zeta)=\zeta$ and $\gamma^{\prime}(\zeta) \neq 1.$  The C$^*$-algebra $C^*(\IP_{\zeta}, U_{\gamma})/\mathcal{K}$ satisfies the hypotheses of Theorem \ref{trajectorial} by the proof of Theorem \ref{p1u}, the discussion preceding Corollary \ref{p1usub}, and the fact that $C^*(\IP_{\zeta})/\mathcal{K} \cong C([0,1])$ is separable.  Thus, we can apply the trajectorial approach to determine the invertibility of elements in this C$^*$-algebra.  More specifically, $b \in C^*(\IP_{\zeta}, U_{\gamma})/\mathcal{K}$ is invertible if and only if $\tau_x(b)$ is an invertible operator on $\ell^2(\IZ)$ for all $x \in [0,1]$, where $\tau_x:=\tau_{ev_x \circ \Gamma_{\zeta}^{-1}}$ is the representation of $C^*(\IP_{\zeta}, U_{\gamma})/\mathcal{K}$ on $\ell^2(\IZ)$ defined according to  (\ref{howtodefinerep}).  
By (\ref{relbtactions}), the representations $\tau_x$ have the form 
\begin{align}\label{defoftaux} \left(\tau_x(\Gamma_{\zeta}(f))h\right)(m)&=f\left(x^{\gamma^{\prime}(\zeta)^m}\right)h(m) \\\ \left(\tau_x([U_{\gamma_n}])h\right)(m)&=h(m+n) \notag \end{align} for all $n, m \in \IZ$, $h \in \ell^2(\IZ)$, $f \in C([0,1])$, and $x \in [0,1]$.

Let $\{\xi_{j}\}_{j \in \IZ}$ be the orthonormal basis of $\ell^2(\IZ)$ that is defined by $\xi_{j}(m)=\delta_{jm}$ for all $j, m \in \IZ$.  Applying (\ref{defoftaux}), we see that if $f \in C([0,1])$, $j, n \in \IZ$, and $x \in [0,1]$, then \begin{align*} \tau_{x}\left(\Gamma_{\zeta}(f)[U_{\gamma_n}]\right)\xi_j=f\left(x^{\gamma^{\prime}(\zeta)^{j-n}}\right)\xi_{j-n}. \end{align*} 
Hence, if \begin{align}\label{formofb} b=\sum_{n=M}^N\Gamma_{\zeta}(f_n)[U_{\gamma_n}] \in C^*(\IP_{\zeta}, U_{\gamma})/\mathcal{K},\end{align} where $M, N \in \IZ$
and $f_n \in C([0,1])$ for all $M \leq n \leq N$, then, for all $x \in [0,1],$ the entries of the matrix of $\tau_x(b)$ 
with respect to the orthonormal basis $\{\xi_j\}_{j \in \IZ}$ are \begin{align}\tau_x(b)_{i,j}=\left<\tau_x(b)\xi_j, \xi_i\right>= \left\{ \begin{array}{ll} f_{j-i}\left(x^{\gamma^{\prime}(\zeta)^i}\right), & M \leq j-i \leq N\\ 0, & \text{otherwise}\end{array}\right. . \end{align}  We will denote this matrix by $[\tau_x(b)].$

For all $x \in [0,1]$, $[\tau_x(b)]$ is a bi-infinite, finite diagonal matrix for which the limits $  \lim_{i \rightarrow \pm \infty} \tau_{x}(b)_{i, i+n}$ exist for all $n \in \IZ$.
  Such operators are called discrete operators and have been studied by Karlovich and Kravchenko in their work on singular integral operators with a shift  \cite{KarlovichKravchenko:1983}.  
As Kravchenko and Litvinchuk have noted, one cannot expect to find efficient conditions for the invertibility of a general discrete operator but instead must
expect only algorithms for checking invertibility conditions that vary in effectiveness depending on the form of the operator \cite[p. 236]{KravchenkoLitvinchuk:1994}.  
Here, we present the algorithm from \cite{KarlovichKravchenko:1983} in the context of the operators $\tau_x(b)$.

If $b$ is of form (\ref{formofb}), we define two functions \begin{align}\label{ourpoly} p_{b, 0}(z)=\sum_{n=M}^N f_{n}(0)z^n \qquad \text{and} \qquad p_{b, 1}(z)=\sum_{n=M}^N f_n(1)z^n. \end{align}
If $\gamma^{\prime}(\zeta)>1$, set $\lambda_+:=0$ and $\lambda_-:=1.$ If $\gamma^{\prime}(\zeta)<1$, then set $\lambda_+:=1$ and $\lambda_-:=0.$ Since $(\lambda_{\pm})^{\gamma^{\prime}(\zeta)^n}
=\lambda_{\pm}$ for all $n \in \IZ$, $[\tau_{\lambda_{\pm}}(b)]$ is a Toeplitz matrix, and $\tau_{\lambda_{\pm}}(b)$ is unitarily equivalent to multiplication by $p_{b, \lambda_{\pm}}$ on $L^2(\IT)$
 via the unitary operator that sends $\xi_j$ to the basis vector $e^{-ij{\theta}} \in L^2(\IT).$  Hence $\tau_{\lambda_{\pm}}(b)$ is invertible if and only if $p_{b, \lambda_{\pm}}$ does not vanish on $\IT.$

If $x \in (0,1)$, then, in general, $[\tau_x(b)]$ is not a Toeplitz matrix, so we need the full theory of discrete operators.   If $p_{b, \lambda_{\pm}}$ does not vanish on $\IT$, then 
we define $\kappa_{b, \lambda_{\pm}}$ to be the winding number of $p_{b, \lambda_{\pm}},$ as a function on $\IT$, around $0$.   If  $\kappa_{b, \lambda_+}=\kappa_{b, 
\lambda_-}=\kappa$, we define, for $x \in (0,1)$ and $\nu, \mu \in \IN$ with $\nu > \mu$, the operators \begin{align*} \tau_x(b)^\nu = \left[\tau_x(b)_{i, j+\kappa}\right]_{i,j=-\nu}^{\nu} \qquad \text{and} \qquad \tau_x(b)^{\nu}_{\mu}=\left[\tau_x(b)_{i, j+\kappa}\right]_{i,j \in J_{\nu, \mu}} \end{align*} where $J_{\nu, \mu}=\{-\nu, \ldots, -\mu, \mu, \ldots, \nu\}.$  Notice that the entries on the $m$th superdiagonal of $\tau_{x}(b)^{\nu}$ are values of $f_{m+\kappa}$.

We can now state the main invertibility theorem for discrete operators from  \cite{KarlovichKravchenko:1983} in the language of our setting.  

\begin{theorem} {\rm \cite[Theorem 5.2]{KarlovichKravchenko:1983}} \label{invworkhorse} Suppose that $\gamma$ is an automorphism of $\ID$ that fixes $\zeta \in \IT$ and satisfies $\gamma^{\prime}(\zeta) \neq 1$ and that $b \in C^*(\IP_{\zeta}, U_{\gamma})/\mathcal{K}$ has form (\ref{formofb}). 
 Define the polynomials $p_{b, \lambda_{\pm}}$ by (\ref{ourpoly}).  Fix $x \in (0,1).$   
 If $\tau_x(b)$ is invertible, then  \begin{align}\label{inv1} p_{b, \lambda_{\pm}}(z) \neq 0 \, \, \text{for all} \, \, z \in \IT\end{align}
  and \begin{align}\label{inv2}\kappa_{b, \lambda_+}=\kappa_{b, \lambda_-}.\end{align} 

If (\ref{inv1}) and (\ref{inv2}) hold, then $\tau_x(b)$ is invertible if and only if there exists $\mu_0 >0$ such that, for all $\mu > \mu_0$, \begin{align}\label{inv3} \omega_{x, \mu}(b) : = \lim_{\nu \rightarrow \infty} \frac{\det \tau_{x}(b)^{\nu}}{\det \tau_x(b)^{\nu}_{\mu}} \neq 0.\end{align} \end{theorem}

Note that, if (\ref{inv1}) and (\ref{inv2}) hold, then the existence of the finite limits $\omega_{x, \mu}(b)$ is guaranteed for sufficiently large $\mu$.

As we expected, condition (\ref{inv3}) is,  in general, difficult to check.  Moreover, to apply it to determine whether an element $b \in C^*(\IP_{\zeta}, U_{\gamma})/\mathcal{K}$ of form (\ref{formofb}) is invertible, we would need to verify (\ref{inv3}) for  infinitely many values of $x$.  
While we cannot fully determine the spectrum of such an element $b$, we obtain a partial spectral result  as a corollary of the preceding theorem.

\begin{corollary}\label{spectralincl} Suppose that $\gamma$ is an automorphism of $\ID$ that fixes $\zeta \in \IT$ and satisfies $\gamma^{\prime}(\zeta) \neq 1$ and that $b \in C^*(\IP_{\zeta}, U_{\gamma})/\mathcal{K}$ has form (\ref{formofb}).  Define the polynomials $p_{b, \lambda_\pm}$ by (\ref{ourpoly}).   Let $\mathcal{W}_b$ be the set of all points $z_0 \in \IC$ such that $p_{b, \lambda_{\pm}}(z)\neq z_0$ for all $z \in \IT$ and $p_{b, \lambda_+}$ and $p_{b, \lambda_-}$ have different winding numbers around $z_0$.  Then $$p_{b, +}(\IT) \cup p_{b, -}(\IT) \cup \mathcal{W}_b \subset \sigma(b).$$ \end{corollary}


\subsection{Spectral Results for Operators in $C^*(C_{\varphi}, \mathcal{K})$}

Let $\varphi$ be a linear-fractional, non-automorphism self-map of $\ID$ that fixes a point $\zeta \in \IT$ and satisfies $\varphi^{\prime}(\zeta) \neq 1.$  By (\ref{adjform})  and Theorem \ref{whatsin}, the coset of any finite product of linear-fractionally-induced composition operators in $C^*(C_{\varphi}, \mathcal{K})$ and their adjoints  is either the coset of the identity operator or has the form $[\alpha C_{\theta}]$, where $\alpha \in \IC$ and $\theta$ is a linear-fractional, non-automorphism self-map of $\ID$ that satisfies $\theta(\zeta)=\zeta$ and $\theta^{\prime}(\zeta)=\varphi^{\prime}(\zeta)^n$ for some $n \in \IZ$.  Thus, by (\ref{howtowritewithu}) and the proof of Lemma \ref{formoffullfixset}, the coset of every finite linear combination of these products can be written in the form \begin{equation}\label{sortofform}\sum_{n=M}^{N} \Gamma_{\zeta}(f_n) \left[U_{\Psi_{\zeta, t^n}}\right], \end{equation} where $M, N \in \IZ$, $t  \in (0, 1) \cup (1, \infty)$, $f_0 \in C([0,1])$, and $f_n \in C_0([0,1])$ for $n \neq 0.$  Recall that $\Psi_{\zeta, t^n} =(\Psi_{\zeta, t})_n$ for all $t >0$ and $n \in \IZ$.  

If $A \in \mathcal{B}(H^2(\ID))$ such that $[A]$ has form (\ref{sortofform}), then  $p_{[A], 0}\equiv f_0(0)$ since $f_n(0)=0$ for $n \neq 0$.  If $f_0(0) \neq 0$, then $\kappa_{[A], 0} = 0.$  Thus, by Corollary \ref{spectralincl}, the essential spectrum of $A$ contains $f_0(0)$, $p_{[A], 1}(\IT)$, and the set of all complex numbers $z_0$ such that winding number of $p_{[A], 1}$ around $z_0$ is defined and non-zero.  

The limit calculations required to fully determine the essential spectrum of $A$ are, in general, still very complicated.  For the remainder of this section, we restrict our attention to the case where $M=0$.  In this setting, the calculations simplify, and we can determine the essential spectrum of $A$.

\begin{theorem}\label{whenFredholm} Suppose $A$ is a bounded linear operator on $H^2(\ID)$ that satisfies $[A]=\sum_{n=0}^N \Gamma_{\zeta}(f_n)\left[U_{\Psi_{\zeta, t^n}}\right]$ for some $\zeta \in \IT$, $t \in (0,1) \cup (1, \infty),$ $N \in \IN \cup \{0\}$, $f_0 \in C([0,1])$, and $f_1, \ldots, f_n \in C_0([0,1]).$  Define the  polynomial $p_{[A], 1}(z):=\sum_{n=0}^Nf_n(1)z^n.$   Then $A$ is Fredholm if and only if $f_0$ does not vanish on $[0,1]$ and $p_{[A], 1}$ has no zeros in $\overline{\ID}$.  Moreover, $$\sigma_e(A)=f_0([0,1]) \cup p_{[A], 1}(\overline{\ID}).$$\end{theorem}  

\begin{proof}
By Theorem \ref{trajectorial}, $A$ is Fredholm if and only if $\tau_{x}([A])$ is invertible for all $x \in [0,1]$.  Thus, by Theorem \ref{invworkhorse} and the discussion preceding it, $A$ is Fredholm if and only $p_{[A],0} \equiv f_0(0)$ and $p_{[A], 1}$ do not vanish on $\IT$, $\kappa_{[A], 1}=\kappa_{[A], 0} = 0$, and, for all $x \in (0,1)$, there exists $\mu_{0,x} >0$ such that $\omega_{x, \mu}([A]) \neq 0$ for all $\mu >\mu_{0, x}.$  Since $p_{[A], 1}$ is a polynomial, the conditions that $p_{[A], 1}$ does not vanish on $\IT$ and $\kappa_{[A], 1}=0$ are equivalent to the single condition that $p_{[A], 1}$ has no zeros in $\overline{\ID}$.  

Let $x \in (0,1)$ and $\nu, \mu \in \IN$ with $\nu > \mu$.  If $\kappa_{[A], 1} = 0$, then $\tau_{x}([A])^{\nu}$ is a $(2\nu +1) \times (2\nu +1)$ upper triangular matrix 
having the entries $f_0\left(x^{t^{-\nu}}\right)$, $f_0\left(x^{t^{-\nu+1}}\right)$, $\ldots$, 
$f_0\left(x^{t^{\nu}}\right)$ along its main diagonal. 
 Similarly, $\tau_{x}([A])_{\mu}^{\nu}$ is a $2(\nu - \mu + 1) \times 2(\nu - \mu +1)$ upper triangular matrix with the entries 
 $f_0\left(x^{t^{-\nu}}\right)$, $f_0\left(x^{t^{-\nu+1}}\right)$, $\ldots$, $f_0\left(x^{t^{-\mu}}\right)$, $f_0\left(x^{t^{\mu}}\right)$, $f_0\left(x^{t^{\mu+1}}\right)$, \ldots $f_0\left(x^{t^{\nu}}\right)$ along the  main diagonal.  
 Hence, for all $x \in (0,1)$ and $\mu \in \IN$,  \begin{align*} \omega_{x, \mu}([A]) = \lim_{\nu \rightarrow \infty} \frac{ \prod_{j=-\nu}^{\nu} f_0\left(x^{t^j}\right)}{\prod_{j=J_{\nu, \mu}} f_0\left(x^{t^j}\right)}= \prod_{j=-\mu+1}^{\mu -1} f_0\left(x^{t^j}\right). \end{align*}  
 Thus, there exists $\mu_{0,x} >0$ such that $\omega_{x, \mu} \neq 0$ for all $\mu > \mu_{0, x}$ if and only if $f_0$ does not vanish on $\left\{x^{t^j}: j \in \IZ\right\}$.
 
 Therefore, $A$ is Fredholm if and only if $f_0(0) \neq 0$, $p_{[A], 1}$ has no zeros in $\ID$, and $f_0$ does not vanish on $(0,1)$.  Since $p_{[A], 1}(0)=f_0(1),$ we obtain the stated result.\end{proof}

We conclude this section with a few examples of operators whose essential spectra  can be calculated by applying Theorem \ref{whenFredholm}.

\begin{example}  Let $n \in \IN$ and suppose that $\varphi_1, \ldots, \varphi_n$ are linear-fractional, non-automorphism self-maps of $\ID$ that fix a point $\zeta \in \IT$ and satisfy $\varphi_1^{\prime}(\zeta) = \ldots = \varphi_n^{\prime}(\zeta)  = s \neq 1.$  By (\ref{howtowritewithu}) and  Lemma \ref{formoffullfixset}, there exist $a_1, \ldots, a_n \in \Upsilon$ such that $[C_{\varphi_j}]=\Gamma_{\zeta} \left(s^{-1/2}x^{a_j} \right)\left[U_{\Psi_{\zeta, s}}\right]$ for all $1 \leq j \leq n.$

If $c_1, \ldots, c_n  \in \IC$, set $\displaystyle{A := \sum_{j=1}^nc_jC_{\varphi_j}}$.  Then $$\left[A \right] = \Gamma_{\zeta}(0)\left[U_{\Psi_{\zeta, s^0}}\right] + \Gamma_{\zeta}\left(s^{-1/2}\sum_{j=1}^n c_jx^{a_j}\right)\left[U_{\Psi_{\zeta, s^1}}\right],$$  so $f_0 \equiv 0$ and $\displaystyle{p_{[A], 1}(z)=\left(s^{-1/2}\sum_{j=1}^nc_j \right)z}$.  Therefore, by Theorem \ref{whenFredholm}, \begin{align*}\sigma_e\left(\sum_{j=1}^nc_jC_{\varphi_j}  \right) = \left\{\lambda \in \IC : |\lambda| \leq s^{-1/2}\left|\sum_{j=1}^n c_j \right|\right\}.\end{align*}  Note that the essential spectral radius of $A$ is equal to the lower bound (\ref{datalowerbound}) on the essential norm of $A$  from \cite{KrieteMoorhouse:2007}.  As a special case of this example, we see that \begin{align*} \sigma_{e}(C_{\varphi_1})=\left\{\lambda \in \IC : |\lambda| \leq (\varphi_1^{\prime}(\zeta))^{-1/2} \right\}. \end{align*} 

\end{example}

\begin{example} Let $\zeta \in \IT$, $a \in \Upsilon,$ and $c_1, c_2 \in \IC.$  Suppose that $\varphi$ is a linear-fractional, non-automorphism self-map of $\ID$ that fixes $\zeta$ and satisfies $\varphi^{\prime}(\zeta) \neq 1.$  Then, by (\ref{howtowritewithu}) and  Lemma \ref{formoffullfixset}, there exist $b \in \Upsilon$ such that 
\begin{align*}[c_1C_{\rho_{\zeta, a}} + c_2C_{\varphi}]=\Gamma_{\zeta}(c_1x^a)\left[U_{\Psi_{\zeta, \varphi^{\prime}(\zeta)^0}}\right] + \Gamma_{\zeta}\left(c_2(\varphi^{\prime}(\zeta))^{-1/2}x^b\right)\left[U_{\Psi_{\zeta, \varphi^{\prime}(\zeta)^1}}\right].\end{align*}Thus, $f_0(x)=c_1x^a$ and $p_{[c_1C_{\rho_{\zeta, a}} + c_2C_{\varphi}], 1}(z)=c_1+\left(c_2(\varphi^{\prime}(\zeta))^{-1/2}\right)z.$  Hence, by Theorem \ref{whenFredholm}, $$\sigma_e(c_1C_{\rho_{\zeta, a}} + c_2C_{\varphi}) = \left\{c_1x^a : x \in [0,1]\right\} \cup \left\{\lambda \in \IC : |\lambda-c_1| \leq |c_2|(\varphi^{\prime}(\zeta))^{-1/2}\right\},$$ the union of a spiral and a closed disk. \end{example}

\section{The Structure of $C^*(T_z, C_{\varphi})/\mathcal{K}$}\label{includeshift}

As we noted in Section \ref{intro}, much of the previous work on  C$^*$-algebras generated by  composition operators has considered C$^*$-algebras that  include the unilateral shift $T_z$ as a generator.  The following two theorems are important results of this type. The inclusion of the unilateral shift as an element of the C$^*$-algebra plays an important role in the proofs of both theorems. 

\begin{theorem}{\rm \cite[Theorem 2.1]{Jury:2007F}}\label{jurymain}  Suppose $\varphi$ is an automorphism of $\ID$.

\begin{enumerate}
\item 
If $\varphi$ has finite order $q$, i.e.\ $\varphi_q(z) \equiv z$ for a positive integer $q$ and $q$ is the smallest positive integer with this property, then $C^*(T_z, C_{\varphi})/\mathcal{K}$ is isometrically $*$-isomorphic to $C(\IT) \rtimes_{\alpha^{\varphi}} \IZ/q\IZ.$
\item Otherwise, $C^*(T_z, C_{\varphi})/\mathcal{K}$ is isometrically $*$-isomorphic to $C(\IT) \rtimes_{\alpha^{\varphi}} \IZ.$  \end{enumerate}  The action $\alpha^{\varphi}$ of $\IZ/q\IZ$, respectively $\IZ$, on $C(\IT)$ is defined by $\alpha^{\varphi}_n(f)=f\circ \varphi_n$ for all $f \in C(\IT).$
\end{theorem}

\begin{theorem}{\rm\cite[Theorem 4.12]{KrieteMacCluerMoorhouse:2007}}\label{distinctpointsmain} Suppose $\varphi$ is a linear-fractional self-map of $\ID$ that is not an automorphism of $\ID$ and satisfies $\varphi(\zeta)=\eta$ for distinct points $\zeta, \eta \in \IT$. Let $\mathcal{D}$ be the C$\, ^*$-subalgebra of $C(\IT) \oplus \mathbb{M}_{2}(C([0,1]))$ defined by $$\mathcal{D}=\left\{(w,V) \in C(\IT) \oplus \mathbb{M}_{2}(C([0,1]))  : V(0) = \left[\begin{array}{cc} w(\zeta) & 0 \\ 0 & w(\eta) \end{array} \right]\right\}.$$  Then  $C^*(T_z, C_{\varphi})/\mathcal{K}$ is isometrically $*$-isomorphic to  $\mathcal{D}.$  \end{theorem}

In the setting of this paper, we did not need to include the unilateral shift to determine the structures of the C$^*$-algebras generated by compact operators and composition operators.  The exclusion of the unilateral shift allows us to focus on the roles of the composition operators in determining the structures of the resulting C$^*$-algebras, but it  obscures the connections between the results in the preceding sections and those in Theorems \ref{jurymain}  and \ref{distinctpointsmain}.  To clarify these relationships, we now include the unilateral shift as a generator and determine the structure of $C^*(T_z, C_{\varphi})/\mathcal{K},$ where $\varphi$ is a linear-fractional, non-automorphism self-map of $\ID$ that fixes a point $\zeta \in \IT$.  When combined with Theorems  \ref{jurymain} and \ref{distinctpointsmain}, these results provide a full characterization of the structures, modulo the ideal of compact operators, of the C$^*$-algebras generated by $T_z$ and a single linear-fractionally-induced composition operator.  Our methods in this section are similar to those used by Kriete, MacCluer, and Moorhouse in \cite{KrieteMacCluerMoorhouse:2007} and \cite{KrieteMacCluerMoorhouse:2009}.

For $t>0$ and $\zeta \in \IT$, we set \begin{align*} N_{\zeta, t} = \left\{\Gamma_{\zeta}(g)\left[U_{\Psi_{\zeta, t^n}}\right] : g \in C_0([0,1]), n \in \IZ\right\}.\end{align*}  Let $\mathcal{A}_{\zeta, t}$ be the non-unital C$^*$-algebra generated by $N_{\zeta, t}$.  By Theorems  \ref{paraboliciso} and \ref{p1usub}, $\mathcal{A}_{\zeta, 1} \cong C_0([0,1])$ and  $\mathcal{A}_{\zeta, t} \cong C_0([0,1]) \rtimes \IZ$ for $t \neq 1.$

If $\varphi$ is a linear-fractional, non-automorphism self-map of $\ID$ that fixes $\zeta$, then \begin{align}\label{genneeded}C^*(T_z, C_{\varphi})/\mathcal{K} = C^*([T_z], N_{\zeta, \varphi^{\prime}(\zeta)}) = C^*(\{[T_f] : f \in C(\IT)\}, \mathcal{A}_{\zeta, \varphi^{\prime}(\zeta)}) \end{align}  by (\ref{formsingle}) and the proof of Theorem \ref{p1usub}. For $t>0$, we define \begin{align*} \mathcal{C}_{\zeta, t}=\left\{[T_f] + [a] : f \in C(\IT), a \in \mathcal{A}_{\zeta, t}\right\}.\end{align*}  It is straightforward to show that $\mathcal{C}_{\zeta, t}$ is closed under addition and adjoints.  By (\ref{genneeded}), we have that  $C^*(T_z, C_{\varphi})/\mathcal{K} = C^*\left(\mathcal{C}_{\zeta, \varphi^{\prime}(\zeta)}\right)$.   We want to show that $\mathcal{C}_{\zeta, \varphi^{\prime}(\zeta)}$ is a C$^*$-algebra, which would imply that $C^*(T_z, C_{\varphi})/\mathcal{K} = \mathcal{C}_{\zeta, \varphi^{\prime}(\zeta)}$.  As a first step toward this goal, we consider the relations between elements of $\mathcal{A}_{\zeta, t}$ and Toeplitz operators   with continuous symbols.

\begin{lemma} \label{multtoepcomp} If $\zeta \in \IT$, $t >0$, $f \in C(\IT)$, and $[a] \in \mathcal{A}_{\zeta, t}$, then \begin{align}\label{tacommute1} [T_f][a]=f(\zeta)[a]=[a][T_f]. \end{align}  Moreover, if $[T_f] + [a]=[0]$, then $f \equiv 0$ and $[a]=0$. \end{lemma}

\begin{proof}
Let $f \in C(\IT)$.  By the proof of Theorem \ref{p1usub}, the finite linear combinations of elements of $N_{\zeta, t}$ form a dense set in $\mathcal{A}_{\zeta, t}$, so it suffices to show that \begin{align}\label{tacommute2} [T_f] \Gamma_{\zeta}(g)[U_{\Psi_{\zeta, t^n}}] = f(\zeta)\Gamma_{\zeta}(g)[U_{\Psi_{\zeta, t^n}}]=\Gamma_{\zeta}(g)[U_{\Psi_{\zeta, t^n}}][T_f]\end{align} for all $g \in C_0([0,1])$ and $n \in \IZ$.  By  Theorem \ref{comprel} and the definition of $\Gamma_{\zeta}$, we see that \begin{align}\label{tacommute3} [T_f]\Gamma_{\zeta}(g)= f(\zeta)\Gamma_{\zeta}(g) = \Gamma_{\zeta}(g)[T_f] \end{align} for all functions $g \in C_0([0,1])$ of the form $g(x)=\sum_{i=1}^m c_ix^{a_i}$.  Since functions of this form are dense in $C_0([0,1])$, (\ref{tacommute3}) and the first equality in (\ref{tacommute2}) hold for all $g \in C_0([0,1])$.

In \cite{Jury:2007G}, Jury proved that, for all automorphisms $\gamma$ of $\ID$,  $[U_{\gamma}][T_f] =[T_{f \circ \gamma}][U_{\gamma}]$.  
Hence \begin{align*} \Gamma_{\zeta}(g)[U_{\Psi_{\zeta, t^n}}][T_f] &= \Gamma_{\zeta}(g)[T_{f \circ \Psi_{\zeta, t^n}}][U_{\Psi_{\zeta, t^n}}]\\ 
&= f(\Psi_{\zeta, t^n}(\zeta))\Gamma_{\zeta}(g)[U_{\Psi_{\zeta, t^n}}] = f(\zeta)\Gamma_{\zeta}(g)[U_{\Psi_{\zeta, t^n}}]\end{align*} 
by (\ref{tacommute3}) and the fact that $\Psi_{\zeta, t^n}$ fixes $\zeta.$  Thus, the second equality in  (\ref{tacommute2}) holds, which proves the first part of the lemma.

Now suppose that $[T_f]+[a]=0$.  Then, by (\ref{tacommute1}) and the properties of Toeplitz operators with continuous symbols, \begin{align*} [0] = [T_f]([T_f]+[a])= [T_{f^2}]+f(\zeta)[a] = [T_{f^2}]-f(\zeta)[T_f] = [T_{f(f-f(\zeta))}].\end{align*}  Since $T_{f(f-f(\zeta))}$ is compact, $f(f-f(\zeta))\equiv 0$.  Hence, for all $z \in \IT$, $f(z)=0$ or $f(z)=f(\zeta)$.  Since $f$ is continuous on $\IT$ and $\zeta \in \IT$, $f(z)=f(\zeta)$ for all $z \in \IT$.  So $[T_f]+[a]=[f(\zeta)I]+[a]$, which is an element of the unitization of $\mathcal{A}_{\zeta, t}$.  Since $[a] \in \mathcal{A}_{\zeta, t}$, $||[f(\zeta)I] + [a]|| \geq |f(\zeta)|.$  Therefore, $f(\zeta)=0$, so $f \equiv 0$ and $[a]=0$. 
\end{proof}

As an immediate consequence of Lemma \ref{multtoepcomp}, we see that $\mathcal{C}_{\zeta, \varphi^{\prime}(\zeta)}$ is closed under multiplication and is thus a dense $*$-subalgebra of $C^*(T_z, C_{\varphi})/\mathcal{K}$.  To show that $\mathcal{C}_{\zeta, \varphi^{\prime}(\zeta)}$ is a C$^*$-algebra,  we set $C_{\zeta}(\IT) = \{f \in C(\IT) : f(\zeta)=0\}$ and let $\mathcal{B}_{\varphi}$ denote the unitization of $C_{\zeta}(\IT) \oplus \mathcal{A}_{\zeta, \varphi^{\prime}(\zeta)}$.  

\begin{theorem} If $\varphi$ is a linear-fractional, non-automorphism self-map of $\ID$ that fixes $\zeta \in \IT$, then $\mathcal{C}_{\zeta, \varphi^{\prime}(\zeta)}$ is a C$^{\, *}$-algebra.  Moreover, $C^*(T_z, C_{\varphi})/\mathcal{K} = \mathcal{C}_{\zeta, \varphi^{\prime}(\zeta)}$, and the map $\Theta_{\varphi} : \mathcal{C}_{\zeta, \varphi^{\prime}(\zeta)} \rightarrow \mathcal{B}_{\varphi}$, which is defined by \begin{align} \Theta_{\varphi} \left([T_f]+[a] \right)= (f-f(\zeta), [a]) +f(\zeta)I\end{align} for all $f \in C(\IT)$ and $[a] \in \mathcal{A}_{\zeta, \varphi^{\prime}(\zeta)}$, is a $*$-isomorphism. \end{theorem}

\begin{proof}  The map $\Theta_{\varphi}$ is well-defined by Lemma  \ref{multtoepcomp}.  Straight-forward calculations show that $\Theta_{\varphi}$ is an injective $*$-homomorphism and that the image of $\mathcal{C}_{\zeta, \varphi^{\prime}(\zeta)}$ under $\Theta_{\varphi}$ is $\mathcal{B}_{\varphi},$ which is a C$^*$-algebra.  Hence $\Theta_{\varphi}^{-1}$ is an injective $*$-homomorphism of $\mathcal{B}_{\varphi}$ into $C^*(T_z, C_{\varphi})/\mathcal{K}$.  Thus $\Theta^{-1}_{\varphi}$ is isometric, and its range $\mathcal{C}_{\zeta, \varphi^{\prime}(\zeta)}$ is closed.  Therefore, $\mathcal{C}_{\zeta, \varphi^{\prime}(\zeta)}$ is a C$^*$-algebra, and the other statements follow immediately. \end{proof}

\begin{corollary}\label{parabolicshift} If $\rho$ is a parabolic, non-automorphism self-map of $\ID$ that fixes the point $\zeta \in \IT$, then $C^*(T_z, C_{\rho})/\mathcal{K}$ is isometrically $*$-isomorphic to the unitization of $C_{\zeta}(\IT) \oplus C_{0}([0,1])$.
\end{corollary}

\begin{corollary}\label{nonparashift} If $\varphi$ is linear-fractional, non-automorphism self-map of $\ID$ that fixes the point $\zeta \in \IT$ and satisfies $\varphi^{\prime}(\zeta) \neq 1,$ then $C^*(T_z, C_{\varphi})/\mathcal{K}$ is isometrically $*$-isomorphic to the unitization of $C_{\zeta}(\IT) \oplus(C_0([0,1]) \rtimes_{\beta^{\varphi}} \IZ),$ where the action $\beta^{\varphi}$ is defined as in Theorem \ref{structureonephi}. \end{corollary}

We note that the unitization of $C_{\zeta}(\IT) \oplus C_{0}([0,1])$ is isometrically $*$-isomorphic to the C$^*$-subalgebra $\tilde{\mathcal{D}}_{\varphi}$ of $C(\IT) \oplus C([0,1])$ defined by $\tilde{\mathcal{D}}_{\varphi} = \{(w,v) \in C(\IT) \oplus C([0,1]) : w(\zeta)=v(0)\}.$  The details of the proof are straight-forward and are left to the reader.  Similarly, the C$^*$-algebra in Corollary \ref{nonparashift} is isomorphic to a C$^*$-subalgebra of the direct sum of $C(\IT)$ and the unitzation of $C_0([0,1]) \rtimes_{\beta^{\varphi}} \IZ$, but we omit the details as they do not add significantly to our understanding of the structure of the C$^*$-algebra.  

We summarize the results of this section in Table \ref{structurecharacterization}.   Note that the first result in the table follows from the fact that $C_{\varphi}$ is compact if $||\varphi||_{\infty} < 1.$

\begin{table}[h]
\centering
\caption{The Structure of $C^*(T_z, C_{\varphi})/\mathcal{K}$ for a Linear-Fractional Self-map $\varphi$ of $\ID$} \label{structurecharacterization}
\begin{tabular}{p{1.75in}p{.05in}p{2.75in}}
\hline &&\\*[-10pt]
Conditions on $\varphi$ & &$C^*(T_z, C_{\varphi})/\mathcal{K}$ is isometrically $*$-isomorphic to \\*[1pt] \hline &&\\*[-8pt]
$||\varphi||_{\infty} < 1$ & &$C(\IT)$\\*[5pt]
automorphism of $\ID$ of finite order $q$ & & $C(\IT) \rtimes_{\alpha^{\varphi}} \IZ/q\IZ$\\*[16pt]
automorphism of $\ID$ of infinite order & & $C(\IT) \rtimes_{\alpha^{\varphi}} \IZ$\\*[16pt]
non-automorphism, $\varphi(\zeta)=\zeta$& &the unitization of $C_{\zeta}(\IT) \oplus C_0([0,1]))$\\  for some $\zeta \in \IT$, $\varphi^{\prime}(\zeta)=1$& &$\cong \{(w, v) \in C(\IT) \oplus C([0,1]) : w(\zeta)=v(0)\}$\\*[5pt]
 non-automorphism, $\varphi(\zeta) = \zeta$ for some $\zeta \in \IT$, $\varphi^{\prime}(\zeta) \neq 1$ && the unitization of $C_\zeta(\IT) \oplus (C_0([0,1]) \rtimes_{\beta^{\varphi}} \IZ)$\\*[16pt]
non-automorphism, $\varphi(\zeta)=\eta$ for some $\zeta, \eta \in \IT$ with $\zeta \neq \eta$ & &\\*[-24pt] && $\left\{\begin{array}{l} (w, V) \in C(\IT) \oplus \mathbb{M}_2(C([0,1])) : \\*[3pt] V(0) = \left(\begin{array}{cc} w(\zeta) & 0 \\ 0 & w(\eta)\end{array} \right)\end{array}\right\}$ \\*[20pt] \hline \end{tabular}
 \end{table} 
 
 \vspace{11pt}

\textit{Acknowledgement.} The author thanks Thomas Kriete for his encouragement and many helpful suggestions during the development of this work.

 \bibliographystyle{abbrv} \renewcommand{\bibname}{Bibliography}  
\bibliography{BiblioFixedPointShort}

\begin{thebibliography}{10}

\bibitem{AntonevichLebedev:1994}
A.~Antonevich and A.~Lebedev.
\newblock {\em Functional-differential equations. {I}. {$C^*$}-theory}.
\newblock Longman Scientific \& Technical, Harlow, 1994.

\bibitem{BasorRetsek:2006}
E.~L. Basor and D.~Q. Retsek.
\newblock Extremal non-compactness of composition operators with linear
  fractional symbol.
\newblock {\em J. Math. Anal. Appl.}, 322(2):749--763, 2006.

\bibitem{Berkson:1981}
E.~Berkson.
\newblock Composition operators isolated in the uniform operator topology.
\newblock {\em Proc. Amer. Math. Soc.}, 81(2):230--232, 1981.

\bibitem{Blackadar:1998}
B.~Blackadar.
\newblock {\em {$K$}-theory for operator algebras}.
\newblock Cambridge University Press, Cambridge, second edition, 1998.

\bibitem{Blackadar:2006}
B.~Blackadar.
\newblock {\em Operator algebras: Theory of $C{^{*}}$-algebras and von Neumann
  algebras}.
\newblock Springer-Verlag, Berlin, 2006.

\bibitem{BottcherSilbermann:2006}
A.~B{\"o}ttcher and B.~Silbermann.
\newblock {\em Analysis of {T}oeplitz operators}.
\newblock Springer-Verlag, Berlin, second edition, 2006.
\newblock Prepared jointly with Alexei Karlovich.

\bibitem{BLNS:2003}
P.~S. Bourdon, D.~Levi, S.~K. Narayan, and J.~H. Shapiro.
\newblock Which linear-fractional composition operators are essentially normal?
\newblock {\em J. Math. Anal. Appl.}, 280(1):30--53, 2003.

\bibitem{BourdonMacCluer:2007}
P.~S. Bourdon and B.~D. MacCluer.
\newblock Selfcommutators of automorphic composition operators.
\newblock {\em Complex Var. Elliptic Equ.}, 52(1):85--104, 2007.

\bibitem{Coburn:1967}
L.~A. Coburn.
\newblock The {$C\sp{\ast} $}-algebra generated by an isometry.
\newblock {\em Bull. Amer. Math. Soc.}, 73:722--726, 1967.

\bibitem{Coburn:1969}
L.~A. Coburn.
\newblock The {$C\sp{\ast} $}-algebra generated by an isometry. {II}.
\newblock {\em Trans. Amer. Math. Soc.}, 137:211--217, 1969.

\bibitem{CowenMacCluer:1995}
C.~C. Cowen and B.~D. MacCluer.
\newblock {\em Composition operators on spaces of analytic functions}.
\newblock CRC Press, Boca Raton, FL, 1995.

\bibitem{Douglas:1973}
R.~G. Douglas.
\newblock {\em Banach algebra techniques in the theory of {T}oeplitz
  operators}.
\newblock American Mathematical Society, Providence, R.I., 1973.

\bibitem{Douglas:1998}
R.~G. Douglas.
\newblock {\em Banach algebra techniques in operator theory}.
\newblock Springer-Verlag, New York, second edition, 1998.

\bibitem{Duren:1970}
P.~L. Duren.
\newblock {\em Theory of {$H\sp{p}$} spaces}.
\newblock Academic Press, New York, 1970.

\bibitem{Jury:2007G}
M.~T. Jury.
\newblock {$C\sp \ast$}-algebras generated by groups of composition operators.
\newblock {\em Indiana Univ. Math. J.}, 56(6):3171--3192, 2007.

\bibitem{Jury:2007F}
M.~T. Jury.
\newblock The {F}redholm index for elements of {T}oeplitz-composition {$C\sp
  *$}-algebras.
\newblock {\em Integral Equations Operator Theory}, 58(3):341--362, 2007.

\bibitem{Karlovich:1988}
Y.~I. Karlovich.
\newblock A local-trajectory method for the study of invertibility in {$C\sp
  *$}-algebras of operators with discrete groups of shifts.
\newblock {\em Dokl. Akad. Nauk SSSR}, 299(3):546--550, 1988.

\bibitem{Karlovich:2007}
Y.~I. Karlovich.
\newblock A local-trajectory method and isomorphism theorems for nonlocal
  {$C\sp *$}-algebras.
\newblock In {\em Modern operator theory and applications}, volume 170 of {\em
  Oper. Theory Adv. Appl.}, pages 137--166. Birkh\"auser, Basel, 2007.

\bibitem{KarlovichKravchenko:1983}
Y.~I. Karlovich and V.~G. Kravchenko.
\newblock The algebra of singular integral operators with piecewise-continuous
  coefficients and piecewise-smooth shift on a complex contour.
\newblock {\em Izv. Akad. Nauk SSSR Ser. Mat.}, 47(5):1030--1077, 1983.

\bibitem{KravchenkoLitvinchuk:1994}
V.~G. Kravchenko and G.~S. Litvinchuk.
\newblock {\em Introduction to the theory of singular integral operators with
  shift}.
\newblock Kluwer Academic Publishers Group, Dordrecht, 1994.
\newblock Translated from the Russian manuscript by Litvinchuk.

\bibitem{KrieteMoorhouse:2007}
T.~Kriete and J.~Moorhouse.
\newblock Linear relations in the {C}alkin algebra for composition operators.
\newblock {\em Trans. Amer. Math. Soc.}, 359(6):2915--2944, 2007.

\bibitem{KrieteMacCluerMoorhouse:2007}
T.~L. Kriete, B.~D. MacCluer, and J.~L. Moorhouse.
\newblock Toeplitz-composition {$C\sp *$}-algebras.
\newblock {\em J. Operator Theory}, 58(1):135--156, 2007.

\bibitem{KrieteMacCluerMoorhouse:2009}
T.~L. Kriete, B.~D. MacCluer, and J.~L. Moorhouse.
\newblock Spectral theory for algebraic combinations of {T}oeplitz and
  composition operators.
\newblock {\em J. Funct. Anal.}, 257(8):2378--2409, 2009.

\bibitem{KrieteMacCluerMoorhousePP}
T.~L. Kriete, B.~D. MacCluer, and J.~L. Moorhouse.
\newblock Composition operators within singly generated composition
  {$C^*$}-algebras.
\newblock {\em Israel J. Math.}, 179:449--477, 2010.

\bibitem{Maunder:1980}
C.~R.~F. Maunder.
\newblock {\em Algebraic topology}.
\newblock Cambridge University Press, Cambridge, 1980.

\bibitem{Pedersen:1979}
G.~K. Pedersen.
\newblock {\em {$C\sp{\ast} $}-algebras and their automorphism groups}.
\newblock Academic Press Inc. [Harcourt Brace Jovanovich Publishers], London,
  1979.

\bibitem{PimsnerVoiculescu:1980}
M.~Pimsner and D.~Voiculescu.
\newblock Exact sequences for {$K$}-groups and {E}xt-groups of certain
  cross-product {$C^{\ast} $}-algebras.
\newblock {\em J. Operator Theory}, 4(1):93--118, 1980.

\bibitem{Rordam:2000}
M.~R{\o}rdam, F.~Larsen, and N.~Laustsen.
\newblock {\em An introduction to {$K$}-theory for {$C^*$}-algebras}.
\newblock Cambridge University Press, Cambridge, 2000.

\bibitem{Shapiro:1993}
J.~H. Shapiro.
\newblock {\em Composition operators and classical function theory}.
\newblock Springer-Verlag, New York, 1993.

\bibitem{ShapiroSundberg:1990}
J.~H. Shapiro and C.~Sundberg.
\newblock Isolation amongst the composition operators.
\newblock {\em Pacific J. Math.}, 145(1):117--152, 1990.

\bibitem{Williams:2007}
D.~P. Williams.
\newblock {\em Crossed products of {$C{\sp \ast}$}-algebras}.
\newblock American Mathematical Society, Providence, RI, 2007.

\end{thebibliography}
\end{document}